\documentclass[12pt,british,a4paper]{amsart}

\usepackage{lmodern}
\usepackage[T1]{fontenc}
\usepackage{microtype}

\usepackage{babel,amssymb,amscd,url,tikz,stmaryrd,mathtools,fixltx2e}
\usepackage{paralist,csquotes,stackrel}
\usepackage[mathscr]{eucal}

\usetikzlibrary{matrix}
\numberwithin{equation}{section}
\mathtoolsset{mathic=true}
\setdefaultenum{\upshape (i)}{}{}{}

\newcommand{\itref}[1]{\eqref{#1}}

\theoremstyle{plain}
\newtheorem{theorem}[equation]{Theorem}
\newtheorem{proposition}[equation]{Proposition}
\newtheorem{lemma}[equation]{Lemma}
\newtheorem{corollary}[equation]{Corollary}

\theoremstyle{definition}
\newtheorem{definition}[equation]{Definition}

\theoremstyle{remark}
\newtheorem{Example}[equation]{Example}
\newtheorem{remark}[equation]{Remark}

\newenvironment{example}{\begin{Example}\pushQED{\qee}}{\popQED\end{Example}}
\DeclareRobustCommand{\qee}{%
  \ifmmode \mathqee
  \else
    \leavevmode\unskip\penalty9999 \hbox{}\nobreak\hfill
    \quad\hbox{\qeesymbol}%
  \fi
}
\newcommand{\mathqee}{\quad\hbox{\qeesymbol}}
\newcommand{\qeesymbol}{\ensuremath\diamondsuit}

\newcommand{\Lie}[1]{\operatorname{\textsl{#1}}}
\newcommand{\lie}[1]{\operatorname{\mathfrak{#1}}}
\newcommand{\GL}{\Lie{GL}}
\newcommand{\gl}{\lie{gl}}
\newcommand{\sln}{\lie{sl}}
\newcommand{\SO}{\Lie{SO}}

\newcommand{\un}{\lie u}

\newcommand{\bmf}{\lie b}

\newcommand{\kf}{\lie k}
\newcommand{\n}{\lie n}
\newcommand{\p}{\lie p}

\newcommand{\tf}{\lie t}

\newcommand{\SL}{\Lie{SL}}
\newcommand{\SU}{\Lie{SU}}
\newcommand{\Un}{\Lie{U}}

\newcommand{\SUr}{\SU(2)_\textup{rotate}}

\newcommand{\ii}{\mathbf i}
\newcommand{\jj}{\mathbf j}

\newcommand{\C}{{\mathbb C}}
\newcommand{\HH}{{\mathbb H}}
\newcommand{\PP}{{\mathbb P}}
\newcommand{\R}{{\mathbb R}}
\newcommand{\Z}{{\mathbb Z}}

\newcommand{\cO}{\mathcal O}
\newcommand{\tH}{\tilde H}
\newcommand{\tT}{\tilde T}
\newcommand{\tmu}{\tilde \mu}

\newcommand{\hks}{\textup{hks}}
\newcommand{\impl}{\textup{impl}}
\newcommand{\reg}{\textup{reg}}
\newcommand{\surj}{\textup{surj}}

\newcommand{\symp}{{\sslash}} 
\newcommand{\hkq}{{\sslash\mkern-6mu/}}

\DeclareMathOperator{\diag}{diag}
\DeclareMathOperator{\dom}{dom}
\DeclareMathOperator{\End}{End}
\DeclareMathOperator{\Hom}{Hom}

\DeclareMathOperator{\im}{im}
\DeclareMathOperator{\LIE}{Lie}
\DeclareMathOperator{\rank}{rank}
\DeclareMathOperator{\Span}{Span}
\DeclareMathOperator{\Spec}{Spec}
\DeclareMathOperator{\Stab}{Stab}
\DeclareMathOperator{\tr}{tr}
\DeclareMathOperator{\vol}{vol}

\DeclarePairedDelimiter{\abs}{\lvert}{\rvert}

\newcommand{\eqbreak}[1][2]{\\&\hskip#1em}

\begin{document}

\title{Implosion for hyperk\"ahler manifolds}

\author{Andrew Dancer}
\address[Dancer]{Jesus College\\
Oxford\\
OX1 3DW\\
United Kingdom} \email{dancer@maths.ox.ac.uk}

\author{Frances Kirwan}
\address[Kirwan]{Balliol College\\
Oxford\\
OX1 3BJ\\
United Kingdom} \email{kirwan@maths.ox.ac.uk}

\author{Andrew Swann}
\address[Swann]{Department of Mathematics\\
Aarhus University\\
Ny Munkegade 118, Bldg 1530\\
DK-8000 Aarhus C\\
Denmark\\
\textit{and}\\
CP\textsuperscript3-Origins,
Centre of Excellence for Cosmology and Particle Physics Phenomenology\\
University of Southern Denmark\\
Campusvej 55\\
DK-5230 Odense M\\
Denmark} \email{swann@imf.au.dk}

\subjclass[2000]{53C26, 53D20, 14L24}

\begin{abstract}
  We introduce an analogue in hyperk\"ahler geometry of the symplectic
  implosion, in the case of $SU(n)$ actions. Our space is a stratified
  hyperk\"ahler space which can be defined in terms of quiver
  diagrams.  It also has a description as a non-reductive geometric
  invariant theory quotient.
\end{abstract}

\maketitle

\section*{Introduction}

Guillemin, Jeffrey and Sjamaar \cite{Guillemin-JS:implosion}
introduced the idea of symplectic implosion of a symplectic manifold
\( M \) with a Hamiltonian action of a compact group \( K \).  The
implosion \( M_\impl \) carries an action of a maximal torus \( T \)
of \( K \), such that the symplectic reductions of \( M \) by \( K \)
agree with the symplectic reductions of the implosion by \( T \). In
this sense the implosion is an abelianisation of the original
Hamiltonian action; the price to be paid for this is that the
implosion is usually quite singular, although it has a stratified
symplectic structure.

The construction of symplectic implosions can be reduced to the
problem of imploding the cotangent bundle \( T^*K \), which thus acts
as a universal implosion.  The imploded space \( (T^*K)_\impl \)
carries a torus action such that the symplectic reductions are the
coadjoint orbits of \( K \). The universal symplectic implosion \(
(T^*K)_\impl \) also has an algebro-geometric description as the
canonical affine completion of the quotient \( K_\C/N \) of the
complexified group \( K_\C \) by a maximal unipotent subgroup~\( N \).

Our aim here is to explore a hyperk\"ahler analogue of the universal
implosion.  In this paper we concentrate on the case of \( \SU(n) \)
actions, where there is a construction involving quiver diagrams,
leaving the case of other compact groups \( K \) to a future paper.
We produce a stratified hyperk\"ahler space \( Q \) whose strata
correspond to quiver diagrams of suitable types. These strata are
hyperk\"ahler manifolds which can be described in terms of open sets
in complex symplectic quotients of the cotangent bundle of \( K_\C =
\SL(n,\C) \) by subgroups containing commutators of parabolic
subgroups.  There is a maximal torus action, and hyperk\"ahler
quotients by this action give the Kostant varieties, which are the
closures in \( \sln(n,\C)^* \) of coadjoint orbits of \( K_\C =
\SL(n,\C) \). We recall that by the work of
Kronheimer~\cite{Kronheimer:nilpotent,Kronheimer:semi-simple},
Biquard~\cite{Biquard} and Kovalev~\cite{Kovalev} all coadjoint orbits
of complex reductive groups admit hyperk\"ahler structures.

We are led to our construction by algebro-geometric considerations,
namely, the wish to produce a variety that is an affine completion of
a complex symplectic quotient of the cotangent bundle \( T^* K_\C \)
by the maximal unipotent subgroup \( N \) of \( K_\C \).  The upshot
is that torus quotients yield not single complex coadjoint orbits but
rather their canonical affine completions which are Kostant varieties.
In particular, the torus reduction at a triple \( (0, \tau_2, \tau_3)
\), where \( \tau_2 + i \tau_3 \) is regular, will give Kronheimer's
hyperk\"ahler structure on the coadjoint orbit of \( \tau_2 + i \tau_3
\) \cite{Kronheimer:semi-simple}. However torus reduction at the
origin yields not a point (as in the symplectic case) but rather the
nilpotent variety.

Given a hyperk\"ahler manifold \( M \) with a hyperk\"ahler
Hamiltonian \( \SU(n) \) action, we can construct its hyperk\"ahler
implosion as the hyperk\"ahler quotient \( (M \times Q) \hkq \SU(n)
\), by analogy with symplectic implosion.  However the connection
between non-abelian and abelian quotients is more involved than in the
symplectic case.  While the torus quotient of the hyperk\"ahler
implosion at a triple \( (0, \tau_2, \tau_3) \) with \( \tau_2 +
i\tau_3 \) regular coincides with Kronheimer's definition of the
hyperk\"ahler reduction of \( M \) by \( \SU(n) \) at this level
\cite{Kronheimer:semi-simple}, reducing at level zero will give the
hyperk\"ahler quotient of the product of \( M \) and the nilpotent
variety.  To recover the usual hyperk\"ahler quotient we must take
just the closed stratum in this space, corresponding to the semisimple
stratum (i.e.\ the point zero) in the nilpotent variety.

We now describe the plan of the paper. In
\S \ref {sec:symplectic-implosion} we briefly review the theory of
symplectic implosion.  In \S \ref {sec:towards-hyperk-impl} we recall
some relevant points from hyperk\"ahler geometry, and introduce
various complex-symplectic spaces which will arise as ingredients for
building the hyperk\"ahler implosion.  In \S \ref {sec:kostant-varieties}
we recall the theory of Kostant varieties and the
Grothendieck-Springer resolution.  Section~\ref {sec:symplectic-quivers} shows
that we may use symplectic quivers associated to actions of products
of special linear groups to give a new model of the symplectic
implosion when \( K = \SU(n) \).  In \S \ref {sec:hyperk-quiv-diagr},
motivated by this construction, we consider hyperk\"ahler quiver
varieties and in \S \ref {sec:strat-quiv-diagr} we stratify them in terms
of quiver diagrams. This gives an approach to hyperk\"ahler implosion
in the case of \( \SU(n) \).  We analyse the structure of these strata
in terms of parabolic subgroups, and identify the implosion with a
non-reductive quotient, in \S \ref {sec:structure-strata}.  Finally in
\S8 we work out various examples and show how Kostant varieties arise
as torus quotients.

\subsubsection*{Acknowledgements.} The work of the second author was
supported by a Senior Research Fellowship of the Engineering and
Physical Sciences Research Council (grant number GR/T016170/1) during
much of this project. The third author is partially supported by the
Danish Council for Independent Research, Natural Sciences. The first
and third authors would like to thank the Banff International Research
Station for its hospitality, during the final stages of this work, at
the workshop \enquote{Geometric Structures on Manifolds}.

\section{Symplectic implosion}
\label{sec:symplectic-implosion}

We first review the theory of symplectic implosion, due to Guillemin,
Jeffrey and Sjamaar \cite{Guillemin-JS:implosion}.  Given a symplectic
manifold \( M \) with a Hamiltonian symplectic action of a compact Lie
group \( K \) with maximal torus \( T \), the imploded space \(
M_\impl \) is a stratified symplectic space with a Hamiltonian action
of the maximal torus \( T \) of \( K \).  For convenience we fix an
invariant inner product on the Lie algebra \( \kf \) of \( K \), which
allows us to identify \( \kf \) with its dual \( \kf^* \), and we fix
a positive Weyl chamber in the Cartan algebra \( \tf \) of \( T \).
Then we have an identification of reduced spaces
\begin{equation}
  M \symp_\lambda^s K   = M_\impl \symp^s_\lambda T
\end{equation}
for all \( \lambda \) in the closure of the fixed positive Weyl
chamber in \( \tf^* \), where \( \symp_\lambda^s \) denotes symplectic
reduction at level \( \lambda \).  Note that \( \lambda \) need not be
central for \( K \); we recall that for general \( \lambda \) the
symplectic reduction \( M \symp_\lambda^s K \) is the space \( (M
\times \mathsf O_{-\lambda}) \symp_0^s K \), where \( \mathsf
O_\lambda \) is the coadjoint orbit of \( K \) through \( \lambda \)
with its canonical symplectic structure. This reduction may be
identified with \( \mu^{-1}(\lambda)/{\Stab_K(\lambda)} \) where \(
\mu\colon M \to \kf^* \) is the moment map for the \( K \)-action on
\( M \) and \( \Stab_K(\lambda) \) is the stabiliser in \( K \) of \(
\lambda \in \kf^* \) under the coadjoint action of \( K \).

A particularly important example of implosion is when we take \( M \)
to be the cotangent bundle \( T^*K \) (which may be identified with \(
K_\C \)). Now \( T^*K \symp_\lambda^s K \) is just the coadjoint orbit
\( \mathsf O_\lambda \) so the imploded space must satisfy
\begin{equation*}
  (T^*K)_\impl\symp_\lambda^s T = \mathsf O_\lambda
\end{equation*}
for \( \lambda \) in the closed positive Weyl chamber \( \tf_{+}^* \).
Explicitly, \( (T^*K)_\impl \) is obtained from \( K \times \tf_{+}^*
\), by identifying \( (k_1, \xi) \) with \( (k_2, \xi) \) if \( k_1,
k_2 \) are related by the translation action of an element of the
commutator subgroup of \( \Stab_K(\xi)\). So if \( \xi \) is in the
interior of the chamber, its stabiliser is a torus and we do not
perform any identifications.  In particular an open dense subset of \(
(T^*K)_\impl \) is just the product of \( K \) with the interior of
the Weyl chamber.

In fact this example gives us a universal imploded space.  As \( T^*K
\) has a Hamiltonian \( K \times K \)-action its implosion inherits a
Hamiltonian \( K \times T \)-action.  Now for a general symplectic
manifold \( M \) with a Hamiltonian \( K \)-action we obtain the
imploded space \( M_\impl \) as the symplectic reduction \( (M \times
(T^*K)_\impl) \symp_0^s K \), which has an induced Hamiltonian \( T
\)-action as required.

Implosion also has an interpretation as an algebro-geometric quotient.
More precisely, \( (T^*K)_\impl \) can be identified with an affine
variety which is the quotient
\begin{equation*}
  K_\C \symp N = \Spec(\cO(K_\C)^N),
\end{equation*}
in the sense of geometric invariant theory (GIT), of the complex
reductive group \( K_\C \) (which is the complexification of \( K \))
by its maximal unipotent subgroup \( N \) \cite{DK,MFK}. This variety
has a stratification by quotients of \( K_\C \) by commutators of
parabolic subgroups; the open stratum is just \( K_\C/N \) and \( K_\C
\symp N \) is the canonical affine completion of the quasi-affine
variety \( K_\C / N \).

\section{Towards hyperk\"ahler implosion}
\label{sec:towards-hyperk-impl}

We now discuss some issues in constructing a hyperk\"ahler analogue of
the symplectic implosion.

\subsection{\relax} In the symplectic case the key example for
implosion was the cotangent bundle \( T^*K \) of a compact Lie group.
It was shown by Kronheimer \cite{Kronheimer:cotangent} that the
cotangent bundle \( T^* K_\C \) of the complexification of \( K \)
admits a hyperk\"ahler structure. (For further aspects of the geometry
of this space, especially the moment geometry,
see~\cite{Dancer-Swann:compact-Lie}).  In fact Kronheimer showed that
\( T^* K_\C \) may be identified with the moduli space \( \mathcal
M_{(K)} \) of solutions to the Nahm equations
\begin{equation*}
  \frac{dT_i}{dt} + [T_0 , T_i] = [T_j, T_k], \qquad \text{\( (ijk) \)
  cyclic permutation of \( (123) \),}
\end{equation*}
(that is, the anti-self-dual-Yang-Mills equations with \( \R^3 \)
translation invariance imposed) where \( T_i \) (for \( i=0,1,2,3 \))
takes values in \( \kf \) and is smooth on the interval \( [0,1] \).
Two solutions are identified if they are equivalent under the gauge
action
\begin{equation*}
  T_0 \mapsto g T_0 g^{-1} - \dot g g^{-1},\qquad 
  T_i \mapsto g T_i g^{-1} \ (i=1,2,3),
\end{equation*}
where \( g \colon [0,1] \mapsto K \) with \( g(0) = g(1) = 1 \in K
\). The Nahm equations may be viewed as the vanishing condition for a
hyperk\"ahler moment map for the action of this group of gauge
transformations on an infinite-dimensional flat quaternionic space of
\( \kf \)-valued functions on \( [0,1] \). In this way \( \mathcal
M_{(K)} \) acquires a hyperk\"ahler structure.  The complex-symplectic
structure defined by the hyperk\"ahler structure on \( \mathcal
M_{(K)} \) is just the standard complex-symplectic form on \( T^* K_\C
\).

We have an action of \( K \times K \) on \( \mathcal M_{(K)} \cong
T^* K_\C \), defined by taking gauge transformations which are no
longer constrained to be the identity at \( t=0,1 \). This action
preserves the hyperk\"ahler structure, and complexifies to a
holomorphic action of \( K_\C \times K_\C \).  Viewing \( T^*K_\C \)
as \( K_\C \times \kf_\C^* \), the left and right actions are given by
\begin{equation}
  \label{eq:action}
  (g, \xi) \mapsto (h_L g h_R^{-1}, Ad(h_R)^* \xi).
\end{equation}

There is also an action of \( \SO(3)) \), given for \( (a_{ij}) \in
\SO(3) \) by
\begin{equation*}
  T_i \mapsto \sum_{j=1}^3 a_{ij}T_j,
\end{equation*} 
which is isometric but rotates the two-sphere of complex structures
associated to the hyperk\"ahler structure, so that all the complex
structures are equivalent.  In terms of the above procedure for
identifying the moduli space with \( T^* K_\C \), the choice of
complex structures corresponds to the choice in the splitting of the
Nahm equations into a real and a complex equation. For the standard
choice of complex structures \( (\mathsf I, \mathsf J, \mathsf K) \)
with corresponding symplectic forms \( (\omega_{\mathsf
I},\omega_{\mathsf J},\omega_{\mathsf K}) \) the complex-symplectic
form is \( \omega_{\mathsf J} + \ii \omega_{\mathsf K} \).

\subsection{\relax} We have already recalled that the symplectic
implosion of \( T^*K \cong K_\C \) can be identified with the
non-reductive GIT quotient \( K_\C \symp N \) and has a stratification
into strata \( K_\C/[P,P] \) where \( P \) ranges over the standard
parabolic subgroups of \( K_\C \). The open stratum, for which \( P \)
is the Borel subgroup \( B \) of \( K_\C \) associated to the choice
of positive Weyl chamber \( \tf_+ \), is \( K_\C/N \) where \( N=[B,B]
\) is a maximal unipotent subgroup of \( K_\C \).

In the hyperk\"ahler setting, it is natural therefore to consider
complex-symplectic quotients of \( T^*K_\C \) by \( N \) and more
generally by commutators of parabolic subgroups \( P \supseteq B \),
acting on the right.

The complex-symplectic moment map for the right action of \( K_\C \)
on \( T^* K_\C \) is just the \( I \)-holomorphic moment map
\begin{equation*}
  (g, \xi) \mapsto \xi,
\end{equation*}
which of course is equivariant for the right action and invariant for
the left action as described in \eqref{eq:action}.  So the
complex-symplectic quotient by the maximal unipotent group \( N \) at
level \( 0 \) is \( K_\C \times_N \n^\circ \), where the annihilator
\( \n^\circ \) in \( \kf_\C^* \) of the Lie algebra \( \n \) of \( N
\) may be identified with the Borel subalgebra \( \bmf \) of \( \kf_\C
\). Here we use a fixed invariant inner product on \( \kf \) to
identify \( \kf \) with \( \kf^* \) and to identify \( \kf_\C \) with
\( \kf_\C^* \); when \( K=\SU(n) \) this identification is given by
the pairing \( (A,B) \mapsto \tr(AB) \) on the Lie algebra of \(
\SL(n,\C) \).

The complex-symplectic form \( \omega_{\mathsf J} + \ii
\omega_{\mathsf K} \) on \( T^* K_\C \) descends to the
complex-symplectic quotient \( K_\C \times_N \n^\circ \). We can also
perform complex-symplectic quotients by commutators \( [P,P] \) of
general parabolics \( P \).  We obtain quotients \( K_\C
\times_{[P,P]} [\p,\p]^\circ \), which may be identified with the
cotangent bundles \( T^* (K_\C/[P,P]) \) of the strata of the
symplectic implosion.

Of course, these quotients carry a complex-symplectic action of \(
K_\C \) induced from the left action on \( T^* K_\C \).  As the
maximal torus \( T_\C \) normalises \( N \) (and \( [P,P] \)) we also
have a surviving right action of \( T_\C \). So these quotients have a
complex-symplectic action of \( K_\C \times T_\C \).

\subsection{\relax} On the other hand the symplectic implosion is
given by the non-reductive GIT quotient \( K_\C \symp N \) which
contains \( K_\C/N \) as an open subset. So when searching for a
candidate for the universal hyperk\"ahler implosion we might look for
a quotient in the sense of GIT of \( K_\C \times \n^\circ \) by the
action of \( N \).  However classical GIT \cite{MFK} only deals with
actions of reductive groups, and the unipotent group \( N \) is not
reductive. There is no difficulty in constructing a non-reductive GIT
quotient \( K_\C \symp N \) of \( K_\C \) by \( N \), since the
algebra \( \cO(K_\C)^N \) of \( N \)-invariant regular
functions on \( K_\C \) is finitely generated and so we can define \(
K_\C \symp N \) to be the associated affine variety
\begin{equation*}
  K_\C \symp N = \Spec(\cO(K_\C)^N).
\end{equation*}
This means that if \( X \) is any complex affine variety on which \(
K_\C \) acts then the algebra of invariants
\begin{equation*}
  \cO(X)^N \cong (\cO(X) \otimes
  \cO(K_\C)^N)^{K_\C} 
\end{equation*}
is finitely generated and we have a non-reductive GIT quotient
\begin{equation*}
  X \symp N = \Spec(\cO(X)^N) \cong (X \times (K_\C \symp N))
  \symp K_\C.   
\end{equation*}
Unfortunately the \( N \) action
\begin{equation*}
  (g, \xi) \mapsto (g n^{-1}, n \xi n^{-1})
\end{equation*}
on \( K_\C \times \n^\circ \) does not extend to a \( K_\C \) action,
so constructing a non-reductive quotient \( (K_\C \times \n^\circ)
\symp N \) is not so straightforward (although see~\cite{DK}).
However we will prove in this paper that when \( K= \SU(n) \) the
algebra \( \cO(K_\C \times \n^\circ)^N \) is finitely
generated, and that
\begin{equation*}
  (K_\C \times \n^\circ ) \symp N = \Spec(\cO(K_\C \times
  \n^\circ)^N) 
\end{equation*}
can be identified with a hyperk\"ahler quotient of a flat space \(
\HH^m \) by a compact group action using quiver diagrams.  This
hyperk\"ahler quotient is a stratified hyperk\"ahler space with a
hyperk\"ahler torus action, and is a complex affine variety for any
choice of complex structure.  It also includes the quotients \( K_\C
\times_{[P,P]} [\p,\p]^\circ \) discussed above, in particular \( K_\C
\times_N \n^\circ \). It is also the canonical affine completion of \(
K_\C \times_N \n^\circ \).

\section{Kostant varieties}
\label{sec:kostant-varieties}

Let us now discuss some links with geometric representation theory
(cf.~\cite{Chriss-G:representation} for background).

We have already observed that we expect the universal hyperk\"ahler
implosion to be a non-reductive GIT quotient
\begin{equation*}
  (K_\C \times \n^\circ) \symp N = \Spec(\cO(K_\C \times
  \n^\circ)^N) 
\end{equation*}
with the actions of the maximal torus \( T \) of \( K \) and its
complexification \( T_\C = B/N \) induced from the action of the Borel
subgroup \( B \) on \( K_\C \times \n^\circ \). We also will be
concerned with the geometric quotient \( K_\C \times_N \n^\circ \).
In this section we shall study quotients and complex-symplectic
reductions of these spaces by tori.

First consider the space \( \tilde{\kf}_\C=K_\C \times_B \n^\circ =
K_\C \times_B \bmf \), where \( B \) is the Borel subgroup with Lie
algebra \( \bmf \).  Now \( \tilde{\kf}_\C \) may be identified via \(
(Q, X) \mapsto (QXQ^{-1}, QB/B) \) with the correspondence space
\begin{equation*}
  \{ (X, \bmf) \in \kf_\C \times {\mathcal B} : X \in \bmf \},
\end{equation*}
where \( {\mathcal B}= K_\C/B \) is the variety of Borel
subalgebras in \( \kf_\C \). Projection onto the second factor
realises \( \tilde{\kf}_\C \) as a vector bundle over \( {\mathcal B}
\).  Projection onto the first factor, on the other hand, gives a map
\begin{gather*}
  \mu \colon \tilde{\kf}_\C = K_\C \times_B \bmf \rightarrow \kf_\C, \\
  \mu \colon (Q, X) \mapsto QXQ^{-1},
\end{gather*}
called the Grothendieck simultaneous resolution.  This map is a closed
and proper surjection (since \( \mathcal B \) is compact).  Over
regular elements of~\( \kf_\C \) it is finite-to-one, of degree~\(
\abs W \), where \( W \) is the Weyl group.

We also have a map
\begin{equation*}
  \rho \colon \kf_\C \rightarrow \C^r \simeq \tf_\C / W,
\end{equation*}
where \( r = \) rank \( \kf_\C \).  This map is defined by choosing
generators \( p_1, \dots, p_r \) for the ring of invariant polynomials
on \( \kf_\C \) and setting
\begin{equation*}
  \rho (X) = (p_1 (X), \dots, p_r(X)).
\end{equation*}

Now let us fix \( X_0 \in \tf_\C \) and consider the subset of \(
\tilde{\kf}_\C \) given by
\begin{equation*}
  \tilde{\kf}_\C (X_0)= K_\C \times_B (X_0 + \n).
\end{equation*}
This has the structure of an affine bundle over \( {\mathcal B} = K_\C
/ B \) (when \( X_0 =0 \) it is the cotangent bundle \( T^*{\mathcal
B} \)).

\begin{lemma}
  \label{lem:muimage}
  We have a surjection
  \begin{equation*}
    \mu \colon  \bigcup_{w \in W}  \tilde{\kf}_\C (w. X_0) \rightarrow
    \rho^{-1} (\chi), 
  \end{equation*}
  where \( \chi = \rho(X_0) \).
\end{lemma}

\begin{proof}
  This follows from the commutative
  diagram~\cite[(3.1.41)]{Chriss-G:representation}:
  \begin{equation*}
    \begin{tikzpicture}
      \matrix [matrix of math nodes, row sep=0.8cm] {
      & |(T)| \tilde{\kf}_\C      &                \\
      |(L)| \kf_\C &                             & |(R)| \tf_\C \\
      & |(B)| \tf_\C / W = \C^r &                \\
      }; \tikzstyle{every node} = [midway,auto,font=\scriptsize]
      \draw[->] (T) -- node {\( \mu \)} (L); \draw[->] (T) --
      node[swap] {\( \nu \)} (R); \draw[->] (L) -- node {\( \rho \)}
      (B); \draw[->] (R) -- node[swap] {\( \pi \)} (B);
    \end{tikzpicture}
  \end{equation*}
  Here \( \nu \) is the map that sends \( X \) to its component \( X_0
  \) in the Cartan algebra, and \( \pi \) is just the quotient by the
  Weyl action.

  Explicitly, we can argue as follows.  If \( p_i \) is an invariant
  polynomial as above, then, letting \( X= X_0 + Y \) where \( Y \in
  \n \), we have
  \begin{equation*}
    p_i \circ \mu(Q, X) = p_i (QXQ^{-1}) = p_i(X) = p_i (X_0)
  \end{equation*}
  where the last equality comes
  from~\cite[Corollary~3.1.43]{Chriss-G:representation}. So \( \mu
  (\tilde{\kf}_\C (X_0)) \) is contained in the fibre \( \rho^{-1}
  (\chi) \) of \( \rho \), where \( \chi = \rho (X_0) \). This
  argument shows \( \mu (\tilde{\kf}_\C (w.X_0)) \) is also contained
  in \( \rho^{-1}(\chi) \), where \( w \) is an element of the Weyl
  group \( W \).

  Conversely, if \( \rho (X) = \chi \), write \( X = \mu (Q, X') = Q
  X' Q^{-1} \) for some \( (Q, X') \in K_\C \times \bmf \).  Write \(
  X' = X_{ss} + Y \), where \( X_{ss} \in \tf_\C \) and \( Y \in \n
  \). As above \( \rho(X_{ss}) = \rho(X') = \rho(X) = \chi = \rho(X_0)
  \), so \( X_{ss}, X_0 \in \tf_\C \) are equivalent under the Weyl
  group action.
\end{proof}

\medbreak We recall some facts, due to Kostant
\cite{Kostant:polynomial}, concerning the {\em Kostant varieties} \(
V_{\chi} = \rho^{-1} (\chi) \)
(see~\cite[\S6.7]{Chriss-G:representation}):
\begin{asparaenum}
\item \( V_{\chi} \) is an irreducible normal affine variety of
  complex dimension \( \dim_\C \kf_\C - r = \dim_\C \kf_\C - \dim
  \tf_\C \),
\item \( V_{\chi} \) is a union of finitely many orbits for the
  adjoint action of \( K_\C \),
\item\label{item:open-dense} there is a unique open dense orbit \(
  V_{\chi}^\reg \), and this consists of the regular elements in \(
  V_{\chi} \),
\item the complement of the open dense orbit \( V_{\chi}^\reg \) is of
  complex codimension \( \geqslant 2 \) in \( V_{\chi} \)
  \cite[Theorem 0.8]{Kostant:polynomial}, and hence \(
  \cO(V_{\chi}^\reg) \) is finitely generated and \( V_{\chi}
  \) is the canonical affine completion \(
  \Spec(\cO(V_{\chi}^\reg)) \) of the quasi-affine variety \(
  V_{\chi}^\reg \),
\item\label{item:closed-orbit}
  there is a unique closed orbit; this consists of the semisimple
  elements in \( V_{\chi} \) and has minimal dimension among the
  orbits in \( V_{\chi} \),
\item \( V_{\chi} \) consists of a single orbit if and only if it
  contains a regular semisimple element.
\end{asparaenum}

Note also that \( \dim_\C \tilde{\kf}_\C(X_0) = 2 \dim_\C \n = \dim_\C
\kf_\C - r = \dim V_{\chi} \).

\bigbreak We know \( \mu \) maps \( \tilde{\kf}_\C (X_0) \) into \(
V_{\chi} \).  But the above discussion of Lemma\nobreakspace \ref {lem:muimage} shows its
image is not contained in a proper subvariety of \( V_{\chi} \). Hence
the image contains a Zariski-open and hence dense set in \( V_{\chi}
\).  But \( \mu \) is closed on \( \tilde{\kf}_\C \) and hence on the
closed subset \( \tilde{\kf}_\C (X_0) \), so the image is all of \(
V_{\chi} \).  This map is in fact injective over regular elements (the
fibre for \( \mu \colon \tilde{\kf}_\C \rightarrow \kf_\C \) over
regulars is finite-to-one with the fibres coming from the Weyl
group). So we have

\begin{lemma}
  \label{lem:mu-resolution}
  We have a surjection
  \begin{equation*}
    \mu \colon \tilde{\kf}_\C (X_0) \rightarrow V_{\chi}
  \end{equation*}
  onto the Kostant variety, which is injective over regular elements.
\end{lemma}

By analogy with symplectic implosion we expect that hyperk\"ahler
quotients of the universal hyperk\"ahler implosion by the maximal
torus \( T \) in \( K \) should be closely related to coadjoint orbits
of the complexified group \( K_\C \). Reduction at \( (0, \zeta_2,
\zeta_3) \) should correspond to the complex-symplectic quotient by \(
T_\C \) at level \( X_0 = \zeta_2 + \ii \zeta_3 \).  So let us
consider the quotients \( \tilde{\kf}_\C(X_0)= K_\C \times_B (X_0 +
\n) \) and also the GIT quotient
\begin{equation}
  \label{eq:hkq}
  (K_\C \times (X_0 + \n)) \symp B.
\end{equation}

From above \( \hat{\mu}\colon (Q, X) \mapsto QXQ^{-1} \) gives a \( B
\)-invariant map from \( K_\C \times (X_0 + \n) \) onto \( V_{\chi} \)
where \( \chi = \rho(X_0) \). This descends to the map \( \mu \colon
\tilde{\kf}_\C(X_0) = K_\C \times_B (X_0 + \n) \rightarrow V_{\chi}
\).

\medbreak If \( X_0 \) is regular semisimple, then, from
\itref{item:closed-orbit} above, \( V_{\chi} \) is just the orbit
through \( X_0 \). Moreover in this case we have an isomorphism
\begin{equation*}
  \mu \colon K_\C \times_B (X_0 + \n) \rightarrow V_\chi,
\end{equation*}
hence the fibres of \( \hat{\mu} \) are exactly the \( B \)-orbits. So
\eqref{eq:hkq} is a geometric quotient and our hyperk\"ahler quotient
is just \( V_{\chi} = K_\C /T_\C \).

\medbreak If \( X_0 \) is not regular semisimple, then \( V_{\chi} \)
will contain more than one orbit. As mentioned above in
\itref{item:open-dense}, the regular elements will form an open dense
set, but the smaller strata in its closure will consist of non-regular
elements.

Now \( \hat{\mu} \colon (Q, X) \mapsto QXQ^{-1} \) still defines a \(
B \)-invariant map from \( K_\C \times (X_0 + \n) \) onto \( V_{\chi}
\), inducing \( \mu \colon \tilde{\kf}_\C(X_0) \rightarrow V_{\chi}
\). This time \( \mu \) is no longer an isomorphism, but it is
injective over the set \( V_{\chi}^\reg \) of regular elements, hence
injective over the complement of a set of codimension \( \geqslant 2
\) in \( V_{\chi} \).

Clearly a polynomial \( f \) on \( V_{\chi} \) induces a \( B
\)-invariant polynomial \( \tilde f \) on \( K_\C \times (X_0 + \n)
\) via \( \tilde f(Q,X) = f(QXQ^{-1}) \).  Conversely, a \( B
\)-invariant polynomial on \( K_\C \times (X_0 + \n) \) will induce a
polynomial on \( \tilde{\kf}_\C(X_0) = K_\C \times_B(X_0 + \n) \)
and hence a polynomial on \( V_{\chi}^\reg \).  This will extend over
the codimension \( \geqslant 2 \) locus of non-regular points to give
a polynomial on \( V_{\chi} \)

Hence \( V_{\chi} \) is the GIT quotient \( (K_\C \times (X_0 + \n))
\symp B \). Thus if, as we expect, the universal hyperk\"ahler
implosion can be identified with a non-reductive GIT quotient \( (K_\C
\times \bmf) \symp N \), then its complex-symplectic GIT reduction by
the torus \( T_\C \) at level \( X_0 \) is \( V_\chi \).

\bigbreak The space \( \tilde{\kf}_\C (X_0) = K_\C \times_B (X_0 +
\n) \), on the other hand, is an affine bundle over \( \mathcal B \)
and is a desingularisation of the Kostant variety \( V_\chi \).

\begin{example}
  In the particular case \( X_0=0 \), then \( \tilde{\kf}_\C (X_0) \)
  is just \( K_\C \times_B \n \), which is the cotangent bundle \(
  T^* {\mathcal B} = T^* (K_\C/B) \). Now the restriction of \( \mu \)
  to \( \tilde{\kf}_\C(0) \) is the Springer resolution
  \begin{equation*}
    \mu \colon T^* {\mathcal B} \rightarrow {\mathcal N},
  \end{equation*}
  where \( {\mathcal N} = V_0 = \rho^{-1}(0) \) is the nilpotent
  variety in \( \kf_\C \).  Both these spaces appear as hyperk\"ahler
  spaces in the work of Nakajima~\cite{Nakajima:ale}.

  If \( K=\SU(2) \) then \( T^*{\mathcal B} \) is \( T^* \PP^1 \), the
  resolution of the nilpotent cone which is the GIT quotient at level
  zero (see \protect \MakeUppercase {E}xample\nobreakspace \ref {ex:SU2} below).
\end{example}

\begin{remark}
  If our candidate for the universal hyperk\"ahler implosion were the
  geometric quotient \( K_\C \times_N \n^\circ = K_\C \times_N \bmf \)
  instead of the non-reductive GIT quotient \( (K_\C \times \bmf)
  \symp N \), then a naive complex-symplectic reduction at level \(
  X_0 \) (taking a geometric rather than GIT quotient) would give the
  Springer resolution \( K_\C \times_B (X_0+\n) =
  \tilde{\kf}_\C(X_0) \) rather than the Kostant variety \( V_\chi \).
\end{remark}

\section{Symplectic quivers}
\label{sec:symplectic-quivers}

In this section we shall present a new model for the universal
symplectic implosion for \(K = \SU(n) \), in terms of \emph{symplectic
quiver representations}.  This will also introduce some ideas which
will be useful in the next section, when we introduce a quiver
description of the universal hyperk\"ahler implosion for \( \SU(n) \).

A \emph{symplectic quiver representation} is a diagram of vector
spaces and linear maps
\begin{equation}
  \label{eq:symplectic}
  0 = V_0 \stackrel{\alpha_0}{\rightarrow}
  V_1 \stackrel{\alpha_1}{\rightarrow}
  V_2 \stackrel{\alpha_2}{\rightarrow} \dots
  \stackrel{\alpha_{r-2}}{\rightarrow} V_{r-1}
  \stackrel{\alpha_{r-1}}{\rightarrow} V_r = \C^n.
\end{equation}
We say that the vector spaces \( V_i \) have dimension vector \(
\mathbf n = (n_1,\dots,n_r ) \) if \( n_i = \dim V_i \). Let
\begin{equation*}
  R(\mathbf n) = \bigoplus_{i=1}^r \Hom(V_{i-1},V_i) 
\end{equation*}
be the space of all such diagrams with \( V_i = \C^{n_i} \) for \( 1
\leqslant i \leqslant r \).  We will say that the representation is
\emph{ordered} if \( 0\leqslant n_1 \leqslant n_2 \leqslant \dots
\leqslant n_r = n \) and \emph{strictly ordered} if \( 0 < n_1
< n_2 < \dots < n_r = n \).

We shall be interested in the GIT quotient of \( R(\mathbf n) \) by an
action of
\begin{equation*}
  \SL \coloneqq \prod_{i=1}^{r-1} \SL(V_i).
\end{equation*}
This is a subgroup of \( \GL \coloneqq \prod_{i=1}^{r-1} \GL(V_i) \)
and for both groups \( g = (g_1,\dots,g_{r-1}) \) acts by
\begin{align*}
  \alpha_i &\mapsto g_{i+1} \alpha_i g_i^{-1} \quad (i = 1,\dots, r-2),\\
  \alpha_{r-1} &\mapsto \alpha_{r-1} g_{r-1}^{-1}.
\end{align*}
There is also of course a commuting action of \( \GL(n,\C) = \GL(V_r)
\) by left multiplication of \( \alpha_{r-1} \).

We shall particularly consider the case of full flag representations
\( \mathbf n = (1,2,\dots,n-1,n) \), but their analysis will require
the study of the more general quiver representations
\eqref{eq:symplectic} for general ordered~\( \mathbf n \).

The points in the GIT quotient \( R(\mathbf n) \symp \SL =
\Spec(\cO(R(\mathbf n)^{\SL})) \) correspond to the closed
orbits for the \( \SL \) action on the affine variety \( R(\mathbf n)
\).  The term \emph{polystable} is often used to describe points of \(
R(\mathbf n) \) which lie in closed \( \SL \)-orbits.  If in addition
the stabiliser is finite, the point is called \emph{stable}.

We first look at length 2 quivers.

\begin{lemma}
  \label{lem:length2}
  A length two diagram
  \begin{equation*}
    V \stackrel{\alpha}{\rightarrow} \C^n
  \end{equation*}
  gives a closed \( \SL(V) \) orbit 
if and only if \( \alpha \) is either \( 0 \)  or injective.
\end{lemma}

\begin{proof}
  Write \( V = \ker\alpha \oplus U \) for some \( U \).  

  Then
  \begin{equation*}
    \alpha =
    \begin{pmatrix}
      0 & a_2
    \end{pmatrix} \colon \ker\alpha \oplus U \to \C^n
  \end{equation*}
  which transforms as
  \begin{equation*}
    \begin{pmatrix}
      0 & a_2
    \end{pmatrix}
    \mapsto
    \begin{pmatrix}
      0 & a_2g_2^{-1}
    \end{pmatrix}
  \end{equation*}
  under the action of \( g= \diag(g_1,g_2) \in \SL(V) \).  If \(
  \ker\alpha\ne 0 \) and \( U\ne 0 \), then \(
  \diag(\lambda^{-k},\lambda^\ell) \in \SL(V) \) where \( \ell = \dim
  \ker\alpha \) and \( k=\dim U \), and the corresponding
  one-parameter group has \( 0 \) as its limit.  Thus, if the orbit is
  closed, either \( \ker\alpha=0 \) and \( \alpha \) is injective or
  \( \alpha = 0 \).

  If \( \alpha =0 \) then the orbit is clearly closed.  If \( \alpha
  \) is injective, consider the action of any one-parameter group of
  \( \SL(V) \) (in the sense of GIT, i.e.\ the image of a homomorphism
  \( \C^* \to \SL(V) \)).  Write \( V = \bigoplus_{i=1}^r E(\mu_i) \)
  as a direct sum of weight spaces.  As the one-parameter group is in
  \( \SL(V) \) we have \( \sum_i\mu_i\dim E(\mu_i) = 0 \).  Now \(
  \alpha = \bigoplus_{i=1}^r a_i \) with each \( a_i \) non-zero and
  under the action of the one-parameter group we have \( a_i \mapsto
  \lambda^{\mu_i} a_i \).  This has a limit as \( \lambda\to\infty \)
  only if \( \mu_i\leqslant0 \) for each \( i \).  But the special
  linear condition then gives \( \mu_i = 0 \) for each~\( i \).  Thus
  by the Mumford numerical criterion~\cite{MFK} \( \alpha \) is stable
  and hence polystable.
\end{proof}

To deal with the general case, we consider the following length~\( 2 \)
situation for a double quiver. Spaces of such double quivers will also
give hyperk\"ahler varieties, and will be studied further in the next
section.

\begin{lemma}
  \label{lem:length2-double}
  A configuration
  \begin{equation}
    \label{eq:length2-double}
    V \stackrel[\beta]{\alpha}{\rightleftarrows} W
  \end{equation}
  of vector spaces \( V,W \) and linear maps \( \alpha , \beta \) gives a
  closed orbit under the \( \SL(V) \)-action
  \begin{equation*}
    \alpha \mapsto \alpha g^{-1}\qquad \beta \mapsto g\beta
  \end{equation*}
  if and only if
  \begin{enumerate}
  \item \( \alpha \) is injective, or
  \item \( \beta \) is surjective, or
  \item \( V = \ker\alpha \oplus \im\beta \).
  \end{enumerate}
\end{lemma}

\begin{proof}
  Suppose \( (\alpha,\beta) \) is \( \SL \)-polystable.  Choose a
  direct sum decomposition \( V = U_1\oplus U_2\oplus U_3 \oplus U_4
  \) such that
  \begin{equation*}
    \ker\alpha \cap \im\beta = U_1,\quad \ker\alpha = U_1 \oplus U_2\quad
    \text{and}\quad\im\beta = U_1 \oplus U_3.
  \end{equation*}
  Then we may write
  \begin{equation*}
    \alpha =
    \begin{pmatrix}
      0&0&a_3&a_4
    \end{pmatrix}
    , \quad
    \beta =
    \begin{pmatrix}
      b_1\\0\\b_3\\0
    \end{pmatrix}
    .
  \end{equation*}
  A block-diagonal element \( \diag(g_1,\dots,g_4) \in \SL(V) \) acts
  as
  \begin{equation*}
    \alpha \mapsto
    \begin{pmatrix}
      0&0&a_3 g_3^{-1}&a_4 g_4^{-1}
    \end{pmatrix}
    , \quad
    \beta \mapsto
    \begin{pmatrix}
      g_1b_1\\0\\g_3b_3\\0
    \end{pmatrix}
    .
  \end{equation*}
  If \( U_2 \) is non-zero, then we may choose \( g_1 \in \GL(U_1) \)
  and \( g_4 \in \GL(U_4) \) freely and take \( g_2 = \lambda \), \(
  g_3 = 1 \), where \( \lambda^k\det g_1\det g_4 = 1 \), and \( k =
  \dim U_2 \).  In particular, we may choose a one-parameter group of
  this form such that \( g_1 \to 0 \) and \( g_4^{-1} \to 0 \).
  Closure of the orbit implies that \( b_1 = 0 \) and \( a_4
  = 0 \).  But \( b_1 \) is onto and \( a_4 \) is injective,
  so \( U_1 = 0 = U_4 \) and \( V = \ker\alpha \oplus \im\beta \).

  If \( U_2 = 0 \) and \( U_1, U_4 \) are both non-zero, then we may
  consider \( g = \diag(g_1,\dots,g_4) \in \SL(V) \) with \( g_3=1 \).
  Now if \( g_1 \to 0 \), we have \( g_4^{-1} \to 0 \) too and closure
  of the orbit implies \( b_1 = 0 = a_4 \), contradicting the
  assumption that \( U_1,U_4 \) are non-zero.

  So we see that \( U_2 = 0 \) implies either \( U_1 = U_2 = 0 \),
  giving \( \alpha \) injective, or \( U_2 = U_4 =0 \), giving \(
  \beta \) surjective.

  We note that the proof of Lemma\nobreakspace \ref {lem:length2} shows that the case
  when \( \alpha \) is injective gives a stable configuration.  For
  the case when \( \beta \) is surjective, note that stability
  of~\eqref{eq:length2-double} is equivalent to stability of the dual
  diagram
  \begin{equation*}
    V^* \stackrel[\alpha^*]{\beta^*}{\rightleftarrows} W^*
  \end{equation*}
  and that \( \beta \) surjective is equivalent to \( \beta^* \)
  injective.

  Finally, for stability of the mixed case when \( V = \ker\alpha
  \oplus \im\beta \), we may assume \( \ker\alpha \) is non-trivial
  and \( \beta \) is not surjective.  Consider the endomorphism \( X =
  \alpha\beta \) of \( W \).  This is invariant under the action of \(
  \SL(V) \) and we have \( \rank X = \rank\alpha = \rank\beta \).

  Suppose \( g_t \in \SL(V) \) is a one-parameter group and that \(
  g_t(\alpha,\beta) \to (\alpha',\beta') \).  We have \(
  \rank\alpha'\leqslant\rank\alpha \) and \(
  \rank\beta'\leqslant\rank\beta \).  But \( \alpha'\beta' = X \), so
  \( \rank X \leqslant \min\{\rank \alpha', \rank \beta'\} \) giving
  \( \rank\alpha' = \rank\alpha \) and \( \rank\beta' = \rank\beta \).
  Also observe that \( \ker \beta' = \ker \beta = \ker X \), by our
  direct sum decomposition.  It now follows that \( \beta' = g\beta \)
  for some \( g \in \GL(V) \).  Since \( \beta \) is not surjective,
  we may in fact choose \( g \in \SL(V) \).  Now \( X = \alpha\beta =
  \alpha'\beta' = \alpha'g\beta \) implies \( \alpha = \alpha' g \) on
  \( \im\beta \).  But \( \alpha \) and \( \alpha' \) have the same
  rank as \( \beta \), and \( \alpha \) is injective on \( \im\beta
  \), so \( \ker(\alpha'g) \) is a complementary subspace to \(
  \im\beta \).  Choosing \( h \in \SL(V) \) extending the identity on
  \( \im\beta \) and with \( h\ker\alpha = \ker(\alpha'g) \), we get
  that \( \alpha' = \alpha (gh)^{-1} \), \( \beta' = gh \beta \) and
  so \( (\alpha',\beta') \) lies in the \( \SL(V) \)-orbit of \(
  (\alpha,\beta) \).
\end{proof}

\begin{theorem}
  \label{thm:symp-stab}
  The symplectic quiver representation \( \alpha \in R(\mathbf n)
  \)~\eqref{eq:symplectic} is \( \SL \)-polystable only if at each
  stage \( i=1,\dots,r-1 \) we have either
  \begin{compactenum}
  \item \( \alpha_i \) is injective, or
  \item \( V_i = \im\alpha_{i-1} \oplus \ker\alpha_i \), or
  \item \( \alpha_{i-1} \) is surjective.
  \end{compactenum}
\end{theorem}

Note that if \( \mathbf n \) is strictly ordered, then the final
possibility cannot occur.

\begin{proof}
  The necessity follows from Lemma\nobreakspace \ref {lem:length2-double} applied to \( V
  = V_i \), \( W = V_{i+1} \oplus V_{i-1} \) and the maps \( \alpha =
  (\alpha_i,0) \), \( \alpha(x) \coloneqq (\alpha_i(x), 0) \) and \(
  \beta = (0,\alpha_{i-1}) \), \( \beta(x,y) \coloneqq \alpha_{i-1}(y)
  \).
\end{proof}

We shall now consider full flag quivers; that is, those where \( V_i =
\C^i \) for \( i=1, \dots, n \).  To describe the GIT quotient by \(
\SL = \prod_{i=2}^{n-1}\SL(i,\C) \) we need to analyse the quotient by
\( \SL \) of the set of such quivers satisfying the conditions given
by Theorem\nobreakspace \ref {thm:symp-stab}; in the course of the analysis we shall see
that the condition of Theorem\nobreakspace \ref {thm:symp-stab} is sufficient as well as
necessary for polystability.

First, observe that we may decompose each vector space \( \C^i \) as
\begin{equation}
  \label{symp:decomp}
  \C^i = \ker \alpha_i \oplus \C^{m_i},
\end{equation}
where \( \C^{m_i} = \C^i \) if \( \alpha_i \) is injective and we take
\( \C^{m_i} = \im \alpha_{i-1} \) otherwise. We put \( m_n = n \).
Note that rank \( \alpha_i = m_i \) for \( 1 \leqslant i \leqslant n-1
\). Further, observe that this actually gives a decomposition of our
quiver into two subquivers, namely
\begin{equation}
  \label{symp:nontriv}
  \C^{m_1} \stackrel{\bar{\alpha}_1}{\rightarrow}
  \C^{m_2} \stackrel{\bar{\alpha}_2}{\rightarrow} \dots
  \stackrel{\bar{\alpha}_{n-2}}{\rightarrow} \C^{m_{n-1}}
  \stackrel{\bar{\alpha}_{n-1}}{\rightarrow} \C^n,
\end{equation}
where \( \bar{\alpha}_i \) denotes the restriction of \( \alpha_i \)
to \( \C^{m_i} \), and the quiver with trivial maps
\begin{equation*}
  \ker \alpha_1 \stackrel{0}{\rightarrow} \ker \alpha_2
  \stackrel{0}{\rightarrow} \dots \stackrel{0}{\rightarrow} \ker
  \alpha_{n-1} \stackrel{0}{\rightarrow} 0.
\end{equation*}
Note that we may use the \( \SL \) action to standardise the
decomposition \eqref{symp:decomp}, and we have a residual action of \(
\prod_{i=1}^{n-1} \Lie S(\GL(m_i,\C) \times \GL(i-m_i,\C)) \)
preserving the decomposition.  The \( \GL(i-m_i,\C) \) action on the
quiver with zero maps is trivial.

If \( \alpha_i \) is not injective, then, since \( \C^i = \ker
\alpha_i \oplus \im \alpha_{i-1} \) we see that \( \rank \alpha_i =
\rank \alpha_{i-1} \).  We deduce that
\begin{alignat}{3}
  \label{eq:mj1}
  m_i &= i, &\quad& \text{if \( \alpha_i \) is  injective,} \\
  \label{eq:mj2}
  m_i &= m_{i-1}, &&  \text{if \( \alpha_i \) is not injective.}
\end{alignat}

Now all the information is contained in the quiver
\eqref{symp:nontriv} with \( \SL(m_i,\C) \) acting for \( m_i = i \)
and \( \GL(m_i,\C) \) acting for \( m_i<i \).  All the restricted maps
\( \bar{\alpha}_i \) are injective, so \( m_{i-1} \leqslant m_i \) for
each \( i \).  When \( m_i = m_{i-1} \), this gives \( m_i \leqslant
i-1 < i \), so we have a \( \GL(m_i,\C) \) action on \( \C^{m_i} \).  We
may use up this action by standardising \( \bar{\alpha}_i \) to be the
identity.  We may therefore remove this edge of the quiver, a process
we call \emph{contraction}.

After performing all such contractions, we arrive at a length \( r \)
quiver where \( m_1 < m_2 < \dots < m_{r-1} < m_r = n \) and all maps
\( \bar{\alpha}_i \) are injective. The residual action is \(
\prod_{i=1}^{r-1} \SL(m_i,\C) \).  For each \( j \), we have \(
\bar{\alpha}_{r-1} \bar{\alpha}_{r-2} \dots \bar{\alpha}_j g_j^{-1} =
(\bar{\alpha}_{r-1} g_{r-1}^{-1}) \prod_{i=r-2}^j g_{i+1}
\bar{\alpha}_i g_i^{-1} \), so if this quiver tends to a limit under
the action of a one-parameter subgroup \( (g_i(t))_{i=1}^{r-1} \),
then the injective map \( \bar{\alpha}_{r-1} \bar{\alpha}_{r-2} \dots
\bar{\alpha}_j g_j^{-1}(t) \) must tend to a limit also. Now the
argument of the last section of Lemma\nobreakspace \ref {lem:length2} shows that \(
g_j(t) \) is trivial.  So there are no destabilising one-parameter
subgroups and we have that the quiver is polystable.

Alternatively, we may use the action of \( \SL(n,\C) \times
\prod_{i=1}^{r-1} \SL(m_i,\C) \) to put the maps \( \bar{\alpha}_i \)
into a form where the only non-zero entries are the \( (j,j) \) terms
for \( j = 1,\dots,m_i \). Moreover the \( (j,j) \) terms may be
chosen to be arbitrary non-zero scalars. It is now straightforward to
check that these scalars may be chosen so that the real moment map for
the action of \( \prod_{i=1}^{r-1} \SU(m_i) \) vanishes, i.e.\ so that
the trace-free part of \( \alpha_{i-1} \alpha_{i-1}^* - \alpha_i^*
\alpha_i \) is zero for \( i=1, \dots, r-1 \). Thus there exists some
\( h \in \SL(n,\C) \) such that the \( \prod_{i=1}^{r-1} \SL(m_i,\C)
\) orbit through the image of the quiver under the action of \( h \)
is closed, and it follows that the orbit through the original quiver
is closed too. So the condition of Theorem\nobreakspace \ref {thm:symp-stab} is sufficient
for polystability.

Our discussion now shows that we have a stratification of the GIT
quotient by \( \prod_{i=2}^{n-1} \SL(i,\C) \) of the space of full flag
quivers. There are \( 2^{n-1} \) strata, corresponding to the strictly
increasing sequences of positive integers ending with \( n \), or
equivalently to the ordered partitions of \( n \). We may also of
course view the strata as being indexed by the standard parabolic
subgroups of \( \SL(n,\C) \).

We next analyse these strata.

\begin{lemma}
  \label{lem:sympstrata}
  If \( 0 < n_1 < n_2 < \dots < n_r = n \) then the space of
  quivers
  \begin{equation*}
    0 \rightarrow V_1\stackrel{\bar{\alpha}_1}{\rightarrow}
    V_2\stackrel{\bar{\alpha}_2}{\rightarrow} \dots
    \stackrel{\bar{\alpha}_{r-2}}{\rightarrow} V_{r-1}
    \stackrel{\bar{\alpha}_{r-1}}{\rightarrow} V_r = \C^n
  \end{equation*}
  with \( \dim V_j = n_j \) and all \( \bar{\alpha}_i \) injective,
  modulo the action of \( \SL = \prod_{i=1}^{r-1} \SL(V_i) \), is
  \begin{equation*}
    \SL(n,\C)/[P,P],
  \end{equation*}
  where \( P \) is the parabolic associated to the flag \( (V_1,
  \dots, V_r) \).

  This statement also holds for quivers
  \begin{equation*}
    0 \leftarrow V_1\stackrel{\beta_1}{\leftarrow}
    V_2\stackrel{\beta_2}{\leftarrow} \dots
    \stackrel{\beta_{r-2}}{\leftarrow} V_{r-1}
    \stackrel{\beta_{r-1}}{\leftarrow} V_r = \C^n,
  \end{equation*}
  where the maps \( \beta_i \colon V_{i+1} \to V_i \) are surjective.
\end{lemma}

\begin{proof}
  We prove this for the case of \( \beta_i \) surjective, and the case
  of \( \alpha_i \) injective follows by dualising.

  We may choose bases for the \( V_i \) so that
  \begin{equation*}
    \beta_i = \left( 0_{n_i \times k_i} \mid I_{n_i \times n_i} \right),
  \end{equation*}
  where \( n_i = \dim V_i \) and \( k_i = n_{i+1} - n_i \) is the
  dimension of the kernel of \( \beta_i \).  Explicitly, once we have
  chosen such a basis \( e_1, \dots , e_{n_i} \) for \( V_i \), we can
  choose a basis \( e_j^{\prime} \), \( 1 \leqslant j \leqslant
  n_{i+1} \), for \( V_{i+1} \) so that the first \( n_{i+1}-n_i \)
  elements are a basis for \( \ker \beta_i \) and for the remaining \(
  n_i \) elements we have \( \beta_i \colon e_j^{\prime} \mapsto e_j
  \). We may view this as using the action of \( \SL \times \SL(n,\C)
  = \prod_{i=1}^r \SL(n_i,\C) \) to standardise the \( \beta_i \).

  How much freedom do we have in choosing such bases? Since \( \beta_i
  \) transforms by \( \beta_i \mapsto g_i \beta_i g_{i+1}^{-1} \), if
  \( \beta_i \) is in the above standard form, then \( \beta_i = g_i
  \beta_i g_{i+1}^{-1} \) if and only if
  \begin{equation*}
    g_{i+1} =
    \begin{pmatrix}
      * & * \\
      0 & g_i
    \end{pmatrix}
    ,
  \end{equation*}
  where the top left block is \( k_i \times k_i \) and the bottom
  right is \( n_i \times n_i \). Moreover \( g_1 \) is an arbitrary
  element of \( \SL(n_1,\C) \).  We see inductively that the freedom
  in \( \SL(n,\C) \) is the commutator of the parabolic group \( P \)
  associated to the flag of dimensions \( (n_1, n_2, \dots, n_r = n)
  \) in \( \C^n \).
\end{proof}

We conclude that the GIT quotient of the space of full flag quivers by
\( \SL \) gives us a description of the symplectic implosion for \(
\SL(n,\C) \).

\begin{theorem}
  \label{thm:sympimp}
  The GIT quotient of the space of full flag quivers by \( \SL =
  \prod_{i=2}^{n-1} \SL(i,\C) \) is the symplectic implosion for \(
  \SU(n) \).  The stratification by quiver diagrams as above
  corresponds to the stratification of the implosion as the disjoint
  union over the standard parabolic subgroups \( P \) of \( \SL(n,\C)
  \) of the varieties \( \SL(n,\C)/[P,P] \).
\end{theorem}

\begin{proof}
  We have already identified the strata of the GIT quotient of the
  space of full flag quivers with the strata of the implosion. Now
  observe that the complement of the open stratum \( \SL(n,\C)/N \)
  (where \( N \) is the maximal unipotent, i.e.\ the commutator of the
  Borel subgroup) is of complex codimension strictly greater than
  one. The universal symplectic implosion and the GIT quotient of the
  space of full flag quivers therefore have the same coordinate ring
  \( \cO(\SL(n,\C))^N \), and as they are both affine varieties
  they are now isomorphic.
\end{proof}

We recall the embedding~\cite{Guillemin-JS:implosion} of the
symplectic implosion
\begin{equation*}
  K_\C \symp N \subset E,
\end{equation*}
where \( E \) is the direct sum of \( K \)-modules \( E = \oplus
V_{\varpi} \), and \( V_{\varpi} \) is the \( K \)-module with highest
weight \( \varpi \).  We take the sum over a minimal generating set
for the monoid of dominant weights. In our case \( K = \SU(n) \) so
these are just the exterior powers of the standard representation \(
\C^n \).  We denote a highest weight vector of \( V_{\varpi} \) by \(
v_{\varpi} \).

Now we can define a map from the space of quivers to the space \( E =
\bigoplus_{j=1}^{n-1} \wedge^j \C^n \) by sending \( \alpha \) to the
element of \( E \) with \( j \)th component
\begin{equation}
  \label{symp:embed}
   \wedge^j(\alpha_{n-1} \dots \alpha_{j+1} \alpha_j) \vol_j \in
   \wedge^j \C^n, 
\end{equation}
where \( \vol_j \) denotes the standard generator of \( \wedge^j \C^j
\).
 
Note that under the action of \( \SL=\prod_{i=1}^{r-1} \SL(i,\C) \), the
composition \( \alpha_{n-1} \dots \alpha_{j+1} \alpha_j \) gets
postmultiplied by \( g_j^{-1} \); but \( g_j \in \SL(j,\C) \), so the \(
j \)th exterior power is invariant. Hence our map descends to the GIT
quotient by \( \SL \), and thus gives an explicit isomorphism of the
GIT quotient to its image \( K_\C \symp N \) in \( E \).

As \( \C^j = \ker \alpha_j \oplus \C^{m_j} \), and the restriction of
\( \alpha_j \) to \( \C^{m_j} \) is injective, we see that \( \wedge^j
(\alpha_{n-1} \dots \alpha_{j+1} \alpha_j) \) is zero if and only if \(
m_j < j \); i.e.\ \( \alpha_j \) is not injective.  Now, from
\eqref{eq:mj1}--\eqref{eq:mj2}, we see that knowing for which indices this
occurs determines the full sequence of \( m_j \), hence which stratum
we are in. So the \( 2^{n-1} \) strata correspond to the possibilities
for which components of \eqref{symp:embed} are zero.

Recall also that the symplectic implosion may be realised as the
closure \( \overline{K_\C v} \), where \( v = \sum v_\varpi \) is the
sum of the highest weight vectors.  Using the Iwasawa decomposition \(
K_\C=KAN \) and recalling that \( N \) fixes \( v \), we see that \(
\overline{K_\C v} = K (\overline{T_\C v}) \), the sweep under the
compact group \( K \) of a toric variety \( \overline{T_\C v} \). In
terms of the quiver model, recall from above we may use the action of
\( \SL(n,\C) \times \prod_{i=1}^{r-1} \SL(m_i,\C) \) to put the maps \(
\bar{\alpha}_i \) into a form where the only non-zero entries are the
\( (j,j) \) terms for \( j = 1, \dots, m_i \), and that these terms
may be set to be arbitrary non-zero scalars. Taking all these scalars
to be \( 1 \), we arrive at the stratification in
Lemma\nobreakspace \ref {lem:sympstrata}. If instead we take all the scalars for \(
\bar{\alpha}_i \) to be equal to a non-zero scalar \( \sigma_i \),
then the freedom in putting \( \bar{\alpha}_i \) into this form is the
parabolic \( P \) rather than its commutator. This gives a description
of the strata as \( (\C^*)^{r-1} \)-bundles over the \emph{compact}
flag variety \( \SL(n,\C)/P \).  This is of course just reflecting
the fact that \( P = [P,P]. T_\C^{r-1} \).  Fixing basepoints in the
flag varieties and taking all the strata together gives the toric
variety \( \overline{T_\C v} \).

\section{Hyperk\"ahler quiver diagrams}
\label{sec:hyperk-quiv-diagr}

We now turn our attention to hyperk\"ahler quiver diagrams.  For \(
K=\SU(n) \) actions this gives us a finite-dimensional approach to
constructing the universal hyperk\"ahler implosion. This uses work on
quiver varieties due to Nakajima~\cite{Nakajima:ale} and Kobak and
Swann \cite{Kobak-S:finite} (see also Bielawski
\cite{Bielawski:hyper-kaehler,Bielawski:GM}). We want to produce a
hyperk\"ahler stratified space with a torus action whose hyperk\"ahler
reductions by this action give Kostant varieties.

Choose integers \( 0 \leqslant n_1 \leqslant n_2 \leqslant \dots
\leqslant n_r=n \) and consider the flat hyperk\"ahler space
\begin{equation}
  \label{eq:Mn}
  M = M(\mathbf n) = \bigoplus_{i=1}^{r-1} \HH^{n_i n_{i+1}} = 
  \bigoplus_{i=1}^{r-1} \Hom(\C^{n_i},\C^{n_{i+1}}) \oplus
  \Hom(\C^{n_{i+1}},\C^{n_i})
\end{equation}
with the hyperk\"ahler action of \( \Un(n_1) \times \dots \times
\Un(n_r) \)
\begin{equation*}
  \alpha_i \mapsto g_{i+1} \alpha_i g_i^{-1},\quad
  \beta_i \mapsto g_i \beta_i g_{i+1}^{-1} \qquad (i=1,\dots r-1),
\end{equation*}
with \( g_i \in \Un(n_i) \) for \( i=1, \dots, r \). Here \( \alpha_i
\) and \( \beta_i \) denote elements of \( \Hom (\C^{n_i},
\C^{n_{i+1}}) \) and \( \Hom (\C^{n_{i+1}},\C^{n_i}) \) respectively,
and right quaternion multiplication is given by
\begin{equation}
  \label{eq:j}
  (\alpha_i,\beta_i)\jj = (-\beta_i^*,\alpha_i^*).
\end{equation}

We may write \( (\alpha,\beta) \in M(\mathbf n) \) as a quiver
diagram:
\begin{equation*}
  0 \stackrel[\beta_0]{\alpha_0}{\rightleftarrows}
  \C^{n_1}\stackrel[\beta_1]{\alpha_1}{\rightleftarrows}
  \C^{n_2}\stackrel[\beta_2]{\alpha_2}{\rightleftarrows}\dots
  \stackrel[\beta_{r-2}]{\alpha_{r-2}}{\rightleftarrows} \C^{n_{r-1}}
  \stackrel[\beta_{r-1}]{\alpha_{r-1}}{\rightleftarrows} \C^{n_r} = \C^n,
\end{equation*}
where \( \alpha_0 = \beta_0 = 0 \).  For brevity, we will often call
such a diagram a quiver.  If each \( \beta_i \) is zero we recover a
symplectic quiver diagram.

Let \( \tH \) be the subgroup, isomorphic to \( \prod_{i=1}^{r-1}
\Un(n_i) \), given by setting \( g_r=1 \) and let
\begin{gather*}
  \tmu \colon M \to \LIE(\tH)
  \otimes \R^3 = 
  \LIE(\tH) \otimes (\R + \C)\\
  \tmu(\alpha,\beta) =
  \bigl((\alpha_i\alpha_i^*-\beta_i^*\beta_i+\beta_{i+1}\beta_{i+1}^*
  - \alpha_{i+1}^*\alpha_{i+1})\ii,
  \alpha_i\beta_i - \beta_{i+1}\alpha_{i+1}\bigr)
\end{gather*}
be the hyperk\"ahler moment map.  Hyperk\"ahler quotients \(
\tmu^{-1}(c)/\tH \) of \( M \) by \( \tH \) (with \( c \in
\LIE(Z(\tH))\otimes \R^3 \) where \( Z(\tH) \cong T^{r-1} \) is the
centre of \( \tH \)) will admit a residual hyperk\"ahler action of \(
\Un(n_r)=\Un(n) \), although in fact only \( \SU(n) \) acts (almost)
effectively as the diagonal central \( \Un(1) \) acts trivially.

It is proved in \cite{Kobak-S:finite} (see also~\cite{KP}) that when
we have a full flag (that is, when \( r=n \) and \( n_j=j \) for each
\( j \), so that the centre of \( \tH \) can be identified with the
maximal torus \( T \) of \( K=\SU(n) \)) then the hyperk\"ahler
quotient \( \tmu^{-1}(0)/\tH \) of \( M \) by \( \tH \) can be
identified with the nilpotent cone in \( \kf_\C \).  Of course the
nilpotent cone in \( \kf_\C \) is the closure of a generic nilpotent
coadjoint orbit. On the other hand it is proved in~\cite{Nakajima:ale}
that cotangent bundles of generalised flag varieties (which are
diffeomorphic to semisimple orbits of \( \sln(n,\C) \)) may be
obtained as hyperk\"ahler quotients \( \tmu^{-1}(c)/\tH \) of
\( M \) by \( \tH \) for generic non-zero \( c \).

Now we may instead reduce in stages by first reducing with respect to
the group \( H = \prod_{i=1}^{r-1}\SU(n_i) \) to obtain a
hyperk\"ahler space \( Q = M \hkq H \), which has a residual action of
the torus \( T^{r-1} \) as well as an action of \( \SU(n_r) = \SU(n)
\), with the hyperk\"ahler quotients of \( Q \) by \( T^{r-1} \)
coinciding with the hyperk\"ahler quotients of~\( M \) by \( \tH \).
This makes the follow definition reasonable.

\begin{definition}
  \label{def:Q}
  The \emph{universal hyperk\"ahler implosion for \( \SU(n) \)} will
  be the hyperk\"ahler quotient \( Q = M \hkq H \), where \( M \), \(
  H \) are as above with \( n_j = j \), for \( j=1, \dots, n
  \), (i.e.\ the case of a full flag quiver).
\end{definition}

Note that as \( Q \) is a hyperk\"ahler reduction by \( H \) at
level~\( 0 \), it inherits an \( \SU(2) \) action that rotates the
two-sphere of complex structures: this action is induced from
multiplication by unit quaternions on~\( M =
\HH^{\sum_{i=1}^{r-1}n_in_{i+1}} \) on the other side from that on
which \( H \) acts and includes the transformation~\eqref{eq:j}.
We denote these group actions on \( M \), \( Q \), etc., by
\begin{equation}
  \label{eq:SU2-rot}
  \SUr.
\end{equation}

Let us take a more detailed look at the structure of \( Q \), in the
general case where the flag is not necessarily full.

The components of the complex moment map \( \mu_\C \) for the \( H \)
action on \( M \) are the trace-free parts of \( \alpha_i \beta_i -
\beta_{i+1} \alpha_{i+1} \) for \( 0 \leqslant i \leqslant r-2 \),
because we are performing a reduction by special unitary rather than
unitary groups.  Hence, the complex moment map equation \( \mu_\C =0
\) can be expressed as the requirement that
\begin{equation}
  \label{eq:mmcomplex}
  \alpha_{i-1}\beta_{i-1} - \beta_i \alpha_i = \lambda^\C_i I \qquad
  (i=1,\dots,r-1),
\end{equation}
for some complex scalars \( \lambda^\C_1,\dots,\lambda^\C_{r-1} \).
Similarly the real moment map equation \( \mu_\R =0 \) can be
expressed as:
\begin{equation}
  \label{eq:mmreal}
  \alpha_{i-1} \alpha_{i-1}^* - \beta_{i-1}^* \beta_{i-1} +
  \beta_i \beta_i^* - \alpha_i^* \alpha_i =
  \lambda^\R_i I \quad (i=1,\dots,r-1),
\end{equation}
where \( \lambda^\R_i \) are real scalars.

The hyperk\"ahler quotient \( Q = M \hkq H = (\mu_\R^{-1}(0) \cap
\mu_\C^{-1}(0))/H \) is the symplectic quotient of the affine variety
\( \mu_\C^{-1}(0) \subset M \) by the compact group~\( H \).  The work
of Kempf and Ness \cite{KN} (cf.~\cite{Kirwan:quotients,Thomas}) shows
that \( Q \) can be canonically identified with the GIT quotient \(
\mu_\C^{-1}(0) \symp H_\C \) of~\( \mu_\C^{-1}(0) \) by the
complexification
\begin{equation*}
  H_\C  = \prod_{i=1}^{r-1}\SL(n_i,\C)
\end{equation*}
of \( H \).  This identification proceeds via the \( H \)-invariant
composition
\begin{equation*}
  \mu_\R^{-1}(0) \cap \mu_\C^{-1}(0) \to \mu_\C^{-1}(0)
  \to \mu_\C^{-1}(0) \symp H_\C
\end{equation*}
of the inclusion of \( \mu_\R^{-1}(0) \cap \mu_\C^{-1}(0) \) in \(
\mu_\C^{-1}(0) \) with the natural map \( \mu_\C^{-1}(0) \to
\mu_\C^{-1}(0) \symp H_\C \).  The action of \( H_\C \) is given by
\begin{gather*}
  \alpha_i \mapsto g_{i+1} \alpha_i g_i^{-1}, \quad \beta_i \mapsto
  g_i \beta_i g_{i+1}^{-1} \qquad (i=1,\dots r-2),\\
  \alpha_{r-1} \mapsto \alpha_{r-1} g_{r-1}^{-1}, \quad \beta_{r-1}
  \mapsto g_{r-1} \beta_{r-1},
\end{gather*}
where \( g_i \in \SL(n_i,\C) \).  Alternatively, one may first perform
the hyperk\"ahler quotient by taking the K\"ahler or GIT quotient \(
\mu_\R^{-1}(0) / H = M \symp H_\C \), and considering the level set
cut out by the image of \( \mu_\C^{-1}(0) \).

In this GIT picture, we have a residual action of \( \SL(n,\C) =
\SL(n_r,\C) \) on the quotient \( Q \) given by
\begin{equation*}
  \alpha_{r-1} \mapsto g_r \alpha_{r-1}, \quad
  \beta_{r-1} \mapsto \beta_{r-1} g_r^{-1}.
\end{equation*}
Explicitly, we see that the action of the \( g_i \) (for \( 1
\leqslant i \leqslant r-1 \)) just conjugates the left-hand side of
\eqref{eq:mmcomplex}, so preserves the equations. This action commutes
with the residual action of \( \tH_\C/H_\C \) which we can
identify with \( (\C^*)^{r-1} \) (and with the maximal torus \( T_\C
\) of \( K_\C \) in the case of a full flag, via the basis of \( \tf
\) given by the simple roots).  Again, we note that if all \( \beta_i
\) are zero then the complex moment map equations hold trivially and
we recover the symplectic quiver situation of
\S \ref {sec:symplectic-quivers}.

For each quiver diagram \( (\alpha,\beta) \in M(\mathbf n) \), we define
\begin{equation*}
  X = \alpha_{r-1} \beta_{r-1} \in \Hom (\C^n,\C^n).
\end{equation*}
This element is invariant under the action of \( \tH_\C \) and transforms
by conjugation under the residual \( \GL(n,\C) = \GL(n_r,\C) \)
action.  In fact, it is the complex-symplectic moment map for the
action of \( \GL(n,\C) \) on~\( M(\mathbf n) \).  We get an \(
(\C^*)^{r-1} \)-invariant and \( \SL(n,\C) \)-equivariant map \( Q
\rightarrow \sln(n,\C) \) induced by:
\begin{equation}
  \label{eq:X0}
  (\alpha, \beta) \mapsto (X)_0 = X - \frac1n \tr(X) I_n
\end{equation}
where \( I_n \) is the \( n\times n \)-identity matrix, which is the
complex-symplectic moment map for the action of \( \SL(n,\C) \) on~\(
M(\mathbf n) \).

We start by obtaining information on the eigenvalues of \( X \) and
other endomorphisms derived from final segments of the
quiver.  Define
\begin{equation}
  \label{eq:defXk}
  X_k = \alpha_{r-1} \alpha_{r-2} \dots
  \alpha_{r-k} \beta_{r-k} \dots \beta_{r-2} \beta_{r-1} \qquad (1
  \leqslant k \leqslant r-1)
\end{equation}
so that \( X = X_1 \).

\begin{lemma}
  For \( (\alpha,\beta) \in \mu_\C^{-1}(0) \),
  satisfying~\eqref{eq:mmcomplex}, we have
  \begin{equation}
    \label{eq:Xformula}
    X_k X = X_{k+1} - (\lambda_{r-1}^\C + \dots +
    \lambda_{r-k}^\C) X_k. 
  \end{equation}
\end{lemma}

\begin{proof}
  This is a straightforward consequence of the equations
  \eqref{eq:mmcomplex}.  We have
  \begin{equation*}
    \begin{split}
      X_k X &= \alpha_{r-1} \dots \alpha_{r-k} \beta_{r-k} \dots
      \beta_{r-2}
      \beta_{r-1} \alpha_{r-1} \beta_{r-1} \\
      &= \alpha_{r-1} \dots \alpha_{r-k} \beta_{r-k} \dots \beta_{r-2}
      (\alpha_{r-2} \beta_{r-2} - \lambda_{r-1}^\C) \beta_{r-1}\\
      &= \alpha_{r-1} \dots \alpha_{r-k} \beta_{r-k} \dots \beta_{r-2}
      \alpha_{r-2} \beta_{r-2} \beta_{r-1} - \lambda_{r-1}^\C X_k.
    \end{split}
  \end{equation*}
  We repeat this process, using the equations successively to shuffle
  the \( \alpha \) term from \( X \) forward until it meets the other
  \( \alpha \)'s.  Each such shuffle means we pick up a \(
  -\lambda_j^\C X_k \) term. After \( k \) such operations, we have
  the desired result.
\end{proof}

\noindent
Putting \( \nu_i = \sum_{j=i}^{r-1} \lambda_j^\C \), so that \( X_k (X
+ \nu_{r-k}) = X_{k+1} \), we find inductively
\begin{equation}
  \label{eq:ann-poly}
  X (X + \nu_{r-1}) \dots (X + \nu_1) = 0.
\end{equation}
We thus have an annihilating polynomial for \( X \) in terms of the \(
\lambda_i^\C \); in particular, if all \( \lambda_i^\C \) are zero,
then \( X \) is nilpotent (cf.~\cite{Kobak-S:finite}).

To gain more information, we decompose \( \C^n \) as the direct sum \(
\bigoplus_{j=1}^\ell \ker (X - \tau_j I)^{m_j} \), where
\begin{equation*}
  (x - \tau_1)^{m_1} \dots (x -\tau_\ell)^{m_\ell},
\end{equation*}
is the characteristic polynomial of~\( X \) and the \( \tau_j \) are
distinct.  More generally, for each~\( i \), we decompose
\begin{equation}
  \label{eq:geneigen}
  \C^{n_i} = \bigoplus_{j=1}^{\ell_i} \ker (\alpha_{i-1} \beta_{i-1} -
  \tau_{i,j} I)^{m_{ij}} 
\end{equation}
where the \( \tau_{i,j} \) for \( 1 \leqslant j \leqslant \ell_i \) are
the eigenvalues of \( \alpha_{i-1} \beta_{i-1} \), with associated
generalised eigenspaces given by the summands on the right hand side
of \eqref{eq:geneigen}.

Since \( (\alpha,\beta) \in \mu_\C^{-1}(0) \),
equation~\eqref{eq:mmcomplex} shows that \( \beta_i (\alpha_i \beta_i
- \tau I) = (\beta_i\alpha_i - \tau I)\beta_i = (\alpha_{i-1}
\beta_{i-1} - (\lambda_i^\C + \tau) I) \beta_i \), for any scalar \(
\tau \).  Thus
\begin{equation*}
  \beta_i (\alpha_i \beta_i - \tau I)^m
  = (\alpha_{i-1} \beta_{i-1} - (\lambda_i^\C + \tau) I)^m \beta_i,
\end{equation*}
for each~\( m \), and \( \beta_i \) restricts to a map
\begin{equation}
  \label{eq:betai}
  \beta_i \colon \ker (\alpha_i \beta_i - \tau I)^m \rightarrow
  \ker(\alpha_{i-1} \beta_{i-1} - (\lambda_i^\C + \tau)I)^m. 
\end{equation}
A similar calculation shows that \( \alpha_i \) restricts to a map
\begin{equation}
  \label{eq:alphai}
  \alpha_i \colon \ker (\alpha_{i-1} \beta_{i-1} - (\lambda_i^\C +
  \tau) I)^m \rightarrow \ker (\alpha_i \beta_i - \tau I)^m. 
\end{equation}
We obtain an endomorphism \( \alpha_i \beta_i \) of \( \ker (\alpha_i
\beta_i - \tau I)^m \), which is an isomorphism unless \( \tau=0 \).
Similarly the composition \( \beta_i \alpha_i \) is an isomorphism of
\( \ker (\alpha_{i-1} \beta_{i-1} - (\lambda_i^\C + \tau)I)^m =\ker
(\beta_i \alpha_i - \tau I)^m \) onto itself unless \( \tau=0 \).
Therefore the maps \eqref{eq:betai} and \eqref{eq:alphai} are
bijective unless \( \tau = 0 \).

We deduce that \( \tau \neq 0 \) is an eigenvalue of \( \alpha_i
\beta_i \) if and only if \( \tau + \lambda_i^\C \neq \lambda_i^\C \)
is an eigenvalue of \( \alpha_{i-1} \beta_{i-1} \). Moreover \(
\alpha_i \) and \( \beta_i \) define isomorphisms between the
associated generalised eigenspaces.  In addition, if the dimension
vector \( (n_1,\dots,n_r) \) is strictly ordered, \( 0 < n_1 <
n_2 < \dots < n_r = n \), then \( \alpha_i \beta_i \in \End(V_{i+1})
\) has zero as an eigenvalue, and \( \alpha_i \), \( \beta_i \)
restrict to maps between the associated generalised \( 0 \)-eigenspace
and the generalised eigenspace for \( \alpha_{i-1} \beta_{i-1} \)
associated to \( \lambda_i^\C \) (this latter space may of course be
zero).

This gives us the following lemma:

\begin{lemma}
  \label{lem:evals}
  Suppose the dimension vector \( \mathbf n \) is strictly ordered.
  Then for \( (\alpha,\beta) \subset M(\mathbf n) \)
  satisfying~\eqref{eq:mmcomplex}, the trace-free endomorphism \(
  (X)_0 = (\alpha_{r-1}\beta_{r-1})_0 \), defined at
  equation~\eqref{eq:X0}, has eigenvalues \( \kappa_1,\dots,\kappa_r
  \), where
  \begin{equation*}
    \begin{split}
      \kappa_j &= \frac1n\Bigl(n_1 \lambda^\C_1 + n_2 \lambda_2^\C +
      \dots + n_{j-1} \lambda^\C_{j-1} \eqbreak[4] -(n-n_j)
      \lambda_j^\C - (n-n_{j+1})\lambda_{j+1}^\C - \dots -
      (n-n_{r-1})\lambda_{r-1}^\C\Bigr).
    \end{split}
  \end{equation*}
  The eigenvalue \( \kappa_j \) occurs with algebraic multiplicity at
  least \( n_j - n_{j-1} \).  If \( \kappa_1, \dots, \kappa_r \) are
  all distinct the multiplicity of \( \kappa_j \) is exactly \( n_j -
  n_{j-1} \).

  Moreover if \( i\leqslant j \) then
  \begin{equation*}
    \kappa_{j+1} - \kappa_i = \lambda_i^\C + \lambda_{i+1}^\C + \dots +
    \lambda_j^\C .
  \end{equation*}
\end{lemma}

It follows from the argument above that we can use generalised
eigenspaces to decompose our quiver as a direct sum of quivers with
maps \( \alpha_{i,j},\beta_{i,j} \),
\begin{equation*}
  V_i^j \stackrel[\beta_{i,j}]{\alpha_{i,j}}{\rightleftarrows} V_{i+1}^j,
\end{equation*}
satisfying \( \alpha_{i,j} \beta_{i,j} - \beta_{i+1,j} \alpha_{i+1, j}
= \lambda_{i+1}^\C \) and such that \( \alpha_{i,j} \beta_{i,j} \) has
only one eigenvalue \( \tau_{i+1,j} \).  

Suppose for some \( j \) we have that \( \alpha_{k,j} \) and \(
\beta_{k,j} \) are isomorphisms for \( i < k < s \) but not for \( k=i \)
or \( k=s \).  Then \( \tau_{i+1,j} = 0 \), \( \tau_{s+1,j} = 0 \) and
Lemma\nobreakspace \ref {lem:evals} implies that \( \sum_{k=i+1}^s \lambda^\C_k = 0 \).

On the other hand, if \( \tau_{i+1,j} \) is non-zero, then \(
\alpha_{i,j} \) and \( \beta_{i,j} \) are isomorphisms.

\begin{remark}
  \label{rem:contract-a}
  Whenever \( \alpha_{i,j} \) is an isomorphism and \( i < r-1 \),
  equation~\eqref{eq:mmcomplex} implies that \( \beta_{i,j} =
  (\alpha_{i,j})^{-1}(\lambda_{i+1}^\C + \beta_{i+1,j}\alpha_{i+1,j})
  \).  We may now perform a contraction of the subquiver analogous to
  that in the symplectic case, by replacing
  \begin{equation*}
    V_{i-1}^j
    \stackrel[\beta_{i-1,j}]{\alpha_{i-1,j}}{\rightleftarrows}
    V_i^j
    \stackrel[\beta_{i,j}]{\alpha_{i,j}}{\rightleftarrows}
    V_{i+1}^j  
    \stackrel[\beta_{i+1,j}]{\alpha_{i+1,j}}{\rightleftarrows}
    V_{i+2}^j
  \end{equation*}
  with
  \begin{equation*}
    V_{i-1}^j
    \stackrel[\beta_{i-1,j}]{\alpha_{i-1,j}}{\rightleftarrows}
    V_i^j
    \stackrel[(\alpha_{i,j})^{-1}\beta_{i+1,j}]{\alpha_{i+1,j}
    \alpha_{i,j}}{\rightleftarrows} 
    V_{i+2}^j.
  \end{equation*}
  The complex moment map equations for the contracted quiver are now
  satisfied with
  \begin{equation*}
    \alpha_{i-1,j}\beta_{i-1,j} - (\alpha_{i,j})^{-1}\beta_{i+1,j}
    \alpha_{i+1,j}\alpha_{i,j} = \lambda_i^\C + \lambda_{i+1}^\C. 
  \end{equation*}
  Conversely, given \( \alpha_{i,j} \) and \( \lambda_{i+1}^\C \) we
  may recover \( \beta_{i,j} \) from the contracted quiver and reverse
  the process.

  Observe that in the situation described above where \( \alpha_{k,j}
  \) and \( \beta_{k,j} \) are isomorphisms for \( i < k < s \) but
  not for \( k = i \) or \( k = s \), then iterating the above
  procedure and using the relation \( \sum_{k=i+1}^s \lambda^\C_k = 0
  \) implies (suppressing the \( j \) index) that \(
  \alpha_i\beta_i\beta_{i+1}\dots\beta_{s-1} =
  \beta_{i+1}\dots\beta_s\alpha_s \).
\end{remark}

Note that given an identification of \( V_{i+1}^j \) with \( V_i^j \),
we may apply the action of \( \SL(V_{i,j}) \) to set \( \alpha_{i,j}
\) to be a non-zero scalar multiple \( aI \) of the identity.  If we
have a \( \GL(V_{i,j}) \) action available we may set \( a=1 \).

\begin{example}
  \label{ex:contract}
  Suppose a quiver
  \begin{equation*}
    0 \rightleftarrows \C^m \rightleftarrows \C^m \rightleftarrows \dots
    \rightleftarrows \C^m \rightleftarrows 0,
  \end{equation*}
  where there are \( p \) copies of \( \C^m \), satisfies the complex
  moment map equations for \( \prod_{k=1}^p \SL(m,\C) \), with \(
  \sum_{k=1}^p \lambda_k^\C = 0 \) but \( \sum_{k=i}^p \lambda_k^\C
  \neq 0 \) for \( i>1 \), and contracts to the zero quiver
  \begin{equation*}
    0 \rightleftarrows \C^m  \rightleftarrows 0.
  \end{equation*}
  Then it lies in the \( \prod_{k=1}^p \SL(m,\C) \)-orbit of a quiver
  of the form
  \begin{equation*}
    0 \rightleftarrows \C^m
    \stackrel[b_1]{a_1}{\rightleftarrows} \C^m
    \rightleftarrows \dots \rightleftarrows \C^m
    \stackrel[b_{p-1}]{a_{p-1}}{\rightleftarrows} \C^m
    \rightleftarrows 0
  \end{equation*}
  where the maps \( \C^m \rightleftarrows \C^m \) are multiplication
  by scalars \( a_j \) and \( b_j \) satisfying \( a_j b_j =
  \sum_{k=j+1}^p \lambda_k^\C \).
\end{example}

\begin{remark}
  \label{rem:contract-b}
  Suppose \( \alpha_{r-1,j} \) is an isomorphism.  Then we may use the
  complex moment map equation to write \( \beta_{r-1,j} =
  (\alpha_{r-2,j}\beta_{r-2,j}-\lambda^\C_{r-1})\alpha_{r-1,j}^{-1}
  \).  We may contract the right-hand end of the subquiver to get
  \begin{equation*}
    \dotsb \rightleftarrows
    V_{r-3,j}
    \stackrel[\beta_{r-3,j}]{\alpha_{r-3,j}}{\rightleftarrows}
    V_{r-2,j}
    \stackrel[\beta_{r-2,j}(\alpha_{r-1,j})^{-1}]{\alpha_{r-1,j}\alpha_{r-2,j}}{\rightleftarrows}
    V_{r,j}
  \end{equation*}
  satisfying the complex moment map equations at \(
  (\lambda^\C_1,\dots,\lambda^\C_{r-2}) \).
\end{remark}

If we are primarily interested in quivers with strictly ordered
dimension vector so that Lemma\nobreakspace \ref {lem:evals} applies, then
Remarks\nobreakspace \ref {rem:contract-a} and\nobreakspace  \ref {rem:contract-b} imply that for many purposes it
is enough to consider the case when \( \tau_{i+1,j} = 0 \) for all \(
i \); that is, when \( X=\alpha_{r-1}\beta_{r-1} \) is nilpotent and
all \( \lambda^\C_i \) are zero, so that the quiver satisfies the
complex moment map equations \( \tmu_\C = 0 \),
\begin{equation} 
  \label{eq:tmuC}
  \alpha_i \beta_i = \beta_{i+1} \alpha_{i+1},
\end{equation}
for \( \tH_\C = \prod_{i=1}^{r-1}\GL(n_i,\C) \).  In this situation we
have the following result from~\cite{Kobak-S:finite}.

\begin{proposition}
  \label{prop:KobS}
  For any dimension vector \( \mathbf n \), the orbit under the action
  of \( \tH_\C = \prod_{i=1}^{r-1} \GL(n_i,\C) \) of a quiver
  \begin{equation*}
    0=V_0\stackrel[\beta_0]{\alpha_0}{\rightleftarrows}
    V_1\stackrel[\beta_1]{\alpha_1}{\rightleftarrows}
    V_2\stackrel[\beta_2]{\alpha_2}{\rightleftarrows}\dots
    \stackrel[\beta_{r-2}]{\alpha_{r-2}}{\rightleftarrows}V_{r-1}
    \stackrel[\beta_{r-1}]{\alpha_{r-1}}{\rightleftarrows} V_r = \C^n ,
  \end{equation*} 
  with \( \dim V_j = n_j \), which satisfies the complex moment map
  equations~\eqref{eq:tmuC} for \( \tH_\C \), is closed if and only if
  it is the direct sum of a quiver
  \begin{equation*}
    0=V_0^{(*)}
    \stackrel[\beta^{(*)}_0]{\alpha^{(*)}_0}{\rightleftarrows}
    V^{(*)}_1
    \stackrel[\beta^{(*)}_1]{\alpha_1^{(*)}}{\rightleftarrows}
    V^{(*)}_2
    \stackrel[\beta^{(*)}_2]{\alpha^{(*)}_2}{\rightleftarrows}
    \dots
    \stackrel[\beta^{(*)}_{r-2}]{\alpha^{(*)}_{r-2}}{\rightleftarrows}
    V^{(*)}_{r-1}
    \stackrel[\beta^{(*)}_{r-1}]{\alpha^{(*)}_{r-1}}{\rightleftarrows}
    V^{(*)}_r = \C^n,  
  \end{equation*} 
  where \( \alpha^{(*)}_j \) is injective and \( \beta^{(*)}_j \) is
  surjective for \( 1\leqslant j<r \) (so for some \( k \), \(
  V_j^{(*)} = 0 \) for \( 0 \leqslant j \leqslant k \)) and a quiver
  \begin{equation*}
    0 = V^{(0)}_0 \rightleftarrows
    V^{(0)}_1 \rightleftarrows
    V^{(0)}_2 \rightleftarrows \dots
    \rightleftarrows V^{(0)}_{r-1}
    \rightleftarrows V^{(0)}_r = 0 
  \end{equation*} 
  in which all maps are \( 0 \).
\end{proposition}

\begin{proof}
  The arguments of \cite[Theorem~2.1]{Kobak-S:finite} show that for \(
  (\alpha,\beta) \in \tmu_\C^{-1}(0) \) the closed \( \tH_\C \)-orbit
  condition corresponds to having direct sums \( V_i = \ker \alpha_i
  \oplus \im \beta_i \).  The complex moment map
  equations~\eqref{eq:tmuC}, imply that this is a direct sum
  decomposition into subquivers, so the maps are zero on the subquiver
  with \( V_i^{(0)} = \ker \alpha_i \), and have the desired
  injectivity\slash surjectivity on the subquiver with \( V_i^{(*)} =
  \im \beta_i \).
\end{proof}

\begin{remark}
  \label{rem:GLdecomp}
  Let us now decompose a strictly ordered quiver representation into a
  sum of subquivers determined by the generalised eigenspaces of the
  compositions~\( \alpha_i\beta_i \).  These subquivers may be
  contracted as in Remarks\nobreakspace \ref {rem:contract-a} and\nobreakspace  \ref {rem:contract-b} to shorter
  quivers with each \( \lambda^\C_i = 0 \).  The contracted subquivers
  thus satisfy the complex moment map equations~\eqref{eq:tmuC} for
  the groups \( \prod_i \GL(V_{i,j}) \) which correspond in their
  situations to \( \tH_\C \).  They lie in closed orbits for these
  groups which play the role of \( \tH_\C \) \emph{provided that} each
  subquiver lies in a closed orbit for the action of \( \prod_i
  \GL(V_{i,j}) \).  If this is the case, we may now apply
  Proposition\nobreakspace \ref {prop:KobS} to the contracted subquivers and deduce, with the
  help of \protect \MakeUppercase {E}xample\nobreakspace \ref {ex:contract}, that the original quiver is the direct
  sum of a quiver in which every \( \alpha_i \) is injective and every
  \( \beta_i \) is surjective and quivers of the form
  \begin{equation}
    \label{eq:ab}
    0 \rightleftarrows \C^m \stackrel[b_1]{a_1}{\rightleftarrows} \C^m
    \stackrel[b_2]{a_1}{\rightleftarrows} \dots
    \stackrel[b_{p-2}]{a_{p-2}}{\rightleftarrows} \C^m 
    \stackrel[b_{p-1}]{a_{p-1}}{\rightleftarrows} \C^m \rightleftarrows 0,  
  \end{equation}
  where the maps are multiplication by non-zero scalars \( a_j \) and
  \( b_j \) satisfying
  \begin{equation}
    \label{eq:ab-lambda}
    a_j b_j = \sum_{k=j+1}^p \lambda_k^\C.
  \end{equation}
  Unfortunately, if the original quiver lies in a closed \( H_\C
  \)-orbit but not a closed \( \tH_\C \)-orbit, we cannot deduce
  directly that the subquivers lie in closed orbits for the action of
  \( \prod_i \GL(V_{i,j}) \).  However we will see that we can get
  around this difficulty by making suitable choices of complex
  structures.
\end{remark}

\begin{remark}
  \label{rem:residual}
  Note that for a quiver of the form~\eqref{eq:ab} we may use the \(
  \GL(m,\C)^p \) action to set the \( a_i = 1 \), so the \( b_i \) are
  now determined by the equations.  Such a quiver will be left
  invariant by any \( g \in \GL(m,\C)^p \) with \( g_1 = \dots =
  g_p \in \GL(m,\C) \), and hence for each such summand we will
  pick up a residual circle action on the injective\slash surjective
  quiver.
\end{remark}

We introduce some notation that will also be useful later for
describing an augmentation process for quivers.

\begin{definition}
  \label{def:subrelation}
  Let \( S \) be a relation on \( \{1,\dots,r-1\} \).  This is the same
  as a subset of \( \{1,\dots,r-1\} \times \{1,\dots,r-1\} \).

  Such an \( S \) is a \emph{subrelation} of another relation \( S' \)
  if \( (i,j) \in S \) implies \( (i,j) \in S' \) for all \( i,j \in
  \{1,\dots,r-1\} \).
\end{definition}

Note that any \( S \) defines a subrelation \( \leqslant_S \) of \(
\leqslant \) by
\begin{equation*}
  {\leqslant_S} = \{(i,j) \in S: i\leqslant j\}
\end{equation*}
This is the maximal subrelation of~\( \leqslant \) contained in~\( S
\).

\begin{definition}
  \label{def:Hs}
  To any relation \( S \) on \( \{1,\dots,r-1\} \) we associate the
  subtorus \( T_S \) of \( \tT = T^{r-1} = \R^{r-1}/\Z^{r-1}
  \) whose Lie algebra is \( \lie t_S = \Span \{ e_{ij} =
  \sum_{k=i}^j e_k : i \leqslant_S j\} \).  We have an exact sequence
  \begin{equation*}
    1 \longrightarrow H \longrightarrow
    \tH \stackrel{\varphi}{\longrightarrow} 
    \tT \longrightarrow 1,
  \end{equation*}
  where the \( i \)th component of \( \varphi \) is the determinant
  map \( \Un(n_i) \to S^1 \).  We define \( H_S \) to be the pre-image
  \begin{equation*}
    H_S = H_S(\mathbf n) = \varphi^{-1}(T_S).
  \end{equation*}
  In particular, \( H_\varnothing = H \) and \( H_\leqslant = \tH \).
\end{definition}

Given any \( \lambda = (\lambda_1,\dots,\lambda_{r-1}) \in
(\R^3)^{r-1} \), we define a relation \( \leqslant_\lambda \) on \(
\{1,\dots,r-1\} \) by
\begin{equation*}
  i \leqslant_\lambda j \iff (i\leqslant j \ \text{and}\ 
  \sum_{k=i}^j \lambda_k = 0 \ \text{in}\ \R^3 ).
\end{equation*}

\begin{proposition}
  \label{prop:gl-decomp}
  Suppose \( (\alpha,\beta) \in M(\mathbf n) \) satisfies the
  hyperk\"ahler moment map equations
  \eqref{eq:mmcomplex}--\eqref{eq:mmreal} for~\( H \).  Assume that \(
  \mathbf n \) is strictly ordered.  Then the group \( \SUr \) of
  equation~\eqref{eq:SU2-rot} contains an element that moves \(
  (\alpha,\beta) \) to a quiver that is a direct sum of subquivers,
  one, \( (\tilde\alpha,\tilde\beta) \), with all \(
  \tilde\alpha \)'s injective and all \( \tilde\beta \)'s surjective,
  and the others of the scalar form \eqref{eq:ab}, with \( a_i \), \(
  b_i \) non-zero satisfying~\eqref{eq:ab-lambda}.
\end{proposition}

\begin{proof}
  Let \( \lambda_j = (\lambda_j^\R, \lambda_j^\C) \in \R \times \C =
  \R^3 \) be the values from the hyperk\"ahler moment map equations.
  Then
  \begin{equation*}
    (\lambda_1,\dots,\lambda_{r-1}) \in (\R^3)^{r-1} \cong \tf^{r-1}
    \otimes \R^3 
  \end{equation*}
  can be identified with the value of the hyperk\"ahler moment map for
  the action of the centre \( T^{r-1} = Z(\tH) \) of \( \tH \).

  The group \( \SUr \) acts on \( \lambda_k \) as an \( \SO(3)
  \)-rotation.  Applying a generic element we may therefore ensure
  that
  \begin{equation}
    \label{eq:lambda-regular}
    i \leqslant_\lambda j \iff (i \leqslant j \ \text{and} \
    \sum_{k=i}^j \lambda_k^\C = 0).
  \end{equation}
  This implies that the rotated quiver~\( (\alpha,\beta) \) satisfies
  the full hyperk\"ahler moment map equations (at level \( 0 \)) for
  the subgroup \( H_{\leqslant_\lambda} \) of \( \tH \).  Its orbit
  under the action of the complexification~\(
  (H_{\leqslant_\lambda})_\C \) is thus closed.  Decomposing with
  respect to generalised eigenspaces and arguing as in
  Remark\nobreakspace \ref {rem:GLdecomp}, we see that each contracted subquiver is closed
  under the action of \( \prod_i \GL(V_{i,j}) \).  By
  Remark\nobreakspace \ref {rem:GLdecomp}, these subquivers have the claimed form.
\end{proof}

Now we have seen that injective\slash surjective quivers are relevant,
we make the following observation about stabilisers:

\begin{lemma}
  \label{lem:SL-stabiliser}
  If a quiver \( (\alpha,\beta) \in M(\mathbf n) \) has the property
  that for all~\( i \), either \( \alpha_i \) is injective or \(
  \beta_i \) is surjective, then the stabiliser for the \( \tH_\C \)
  action is trivial.
\end{lemma}

\begin{proof}
  If \( g \in \prod_{i=1}^{r-1} \GL(n_i,\C) \) stabilises the quiver,
  then we have
  \begin{equation*}
    g_{i+1} \alpha_i = \alpha_i g_i, \qquad g_i \beta_i = \beta_i g_{i+1},
  \end{equation*}
  for \( i=1, \dots, r-2 \), with \( \alpha_{r-1} = \alpha_{r-1}
  g_{r-1} \) and \( g_{r-1} \beta_{r-1} = \beta_{r-1} \).  Our
  injectivity\slash surjectivity assumption now means we can work
  inductively down from the top of the quiver to deduce each \( g_i \)
  is the identity element.
\end{proof}

\begin{lemma}
  \label{lem:GITs}
  Let \( S \) be any relation on \( \{1,\dots,r-1\} \).  Suppose that
  a quiver \( (\alpha,\beta) \in M(\mathbf n) \) has each \( \alpha_i
  \) injective and each \( \beta_i \) surjective.  Then it is stable
  in the sense of GIT for the action of the complexification \(
  (H_S)_\C \) of \( H_S \).

  In the case \( H_S = H \), this is true under the weaker assumption
  that all \( \alpha_i \) are injective or all \( \beta_i \)
  surjective.
\end{lemma}

\begin{proof}
  Note that, by the argument of Lemma\nobreakspace \ref {lem:SL-stabiliser}, the stabiliser
  in \( \tH_\C = \prod_{i=1}^{r-1} \GL(n_i,\C) \) is trivial.
  Moreover for \( 1 \leqslant k < r \) the maps
  \begin{gather*}
    \wedge^{n_k} (\alpha_{r-1} \alpha_{r-2} \dots
    \alpha_k) \colon \wedge^{n_k} \C^{n_k} \to \wedge^{n_k}
    \C^n\quad\text{and} \\ 
    \wedge^{n_k} (\beta_k \dots  \beta_{r-2} 
    \beta_{r-1}) \colon \wedge^{n_k} \C^n \to \wedge^{n_k} \C^{n_k}
  \end{gather*}
  are both non-zero, and under the action of \( (H_S)_\C \) one of
  them is multiplied by a character, while the other is multiplied by
  its inverse.  This means that these maps are non-zero for every
  quiver in the closure of the \( (H_S)_\C \)-orbit, and hence that
  every quiver in the closure of the \( (H_S)_\C \)-orbit also has
  each \( \alpha_i \) injective and each \( \beta_i \) surjective, and
  hence has trivial stabiliser. So the orbit is closed of maximal
  dimension, as required.

  If \( H_S = H \), this argument works under the weaker assumption as
  the exterior power maps are all invariant under the \( H_\C \)
  action.
\end{proof}

Returning to the endomorphism \( X = \alpha_{r-1}\beta_{r-1} \)
associated to a quiver \( (\alpha,\beta) \), we now note that
\eqref{eq:ann-poly} is actually the minimum polynomial of \( X \) if
we impose appropriate non-degeneracy conditions on the quiver diagram.

\begin{proposition}
  Suppose \( \mathbf n \in \Z^r_{>0} \) is strictly ordered and \(
  (\alpha,\beta) \in M(\mathbf n) \) satisfies the complex moment map
  equations~\eqref{eq:mmcomplex} for~\( H_\C \).  If each \( \alpha_i
  \) is injective and each \( \beta_i \) surjective, then no
  polynomial of degree less than \( r \) annihilates \( X \), and
  hence
  \begin{equation*}
    x (x + \nu_{r-1}) \dots (x + \nu_1)
  \end{equation*}
  is the minimum polynomial of~\( X \).
\end{proposition}

\begin{proof}
  It follows from our formula \eqref{eq:Xformula} that if \( p(x) \)
  is a degree \( k \) monic polynomial then \( p(X) \) is a linear
  combination of \( X_k \), \( X_{k-1} \),\dots, \( X_1 = X \) and \(
  I \), with the coefficient of \( X_k \) being \( 1 \).  We write \(
  p(X) = X_k + c_{k-1} X_{k-1} + \dots + c_1 X + c_0 I \).

  We may choose vectors \( w_i \) in \( V_r \) for \( i=2, \dots, r \)
  such that, for each \( i \), the vector \( w_i \) is killed by \(
  \beta_{i-1} \beta_i \dots \beta_{r-1} \) but not by \( \beta_i \dots
  \beta_{r-1} \). To see this, let \( x_i \) be a non-zero element of
  \( \ker \beta_{i-1} \) and then, using surjectivity of the \(
  \beta_j \), let \( w_i \) satisfy \( \beta_i \dots \beta_{r-1} w_i =
  x_i \).

  If the \( \alpha_i \) are injective, then, recalling the definition
  \eqref{eq:defXk} of \( X_k \), we see that \( X_j \) kills \( w_i \) if
  and only if \( j \geqslant r-i+1 \).

  Now if \( k \leqslant r-1 \) and
  \begin{equation*}
    p(X) = X_k + c_{k-1} X_{k-1} + \dots + c_1 X + c_0 = 0
  \end{equation*}
  then successively applying this equation to \( w_r, w_{r-1}, \dots,
  w_2 \) yields that each of \( c_0, c_1, \dots, c_{k-1} \) is
  zero. Hence \( X_k \) is zero, which is impossible since (using
  injectivity of \( \alpha_i \) and surjectivity of \( \beta_i \)) we
  have that the rank of \( X_k \) equals \( n_{r-k}=\dim V_{r-k} \).
\end{proof}

Actually we can get all the coadjoint orbits by considering quivers of
this type.

\begin{proposition}
  \label{prop:allorbits}
  Every element of \( \sln(n,\C) \) may be obtained from a quiver
  \begin{equation*}
    0 \stackrel[\beta_0]{\alpha_0}{\rightleftarrows}
    \C^{n_1}\stackrel[\beta_1]{\alpha_1}{\rightleftarrows}
    \C^{n_2}\stackrel[\beta_2]{\alpha_2}{\rightleftarrows}
    \dots 
    \stackrel[\beta_{r-2}]{\alpha_{r-2}}{\rightleftarrows}
    \C^{n_{r-1}} 
    \stackrel[\beta_{r-1}]{\alpha_{r-1}}{\rightleftarrows}
    \C^{n_r} = \C^n 
  \end{equation*}
  with \( 0 < n_1 < n_2 < \dots < n_r = n \), all \( \alpha_i \)
  injective, all \( \beta_i \) surjective and satisfying the complex
  moment map equations~\eqref{eq:mmcomplex} for~\( H_\C \).
\end{proposition}

\begin{proof}
  Observe first of all that every element in \( \sln(n,\C) \) may be
  realised as the trace-free part \( (X)_0 \) of some \( X \) in \(
  \gl(n,\C) \) with a zero eigenvalue: if \( \lambda \) is an
  eigenvalue of \( Y \in \sln(n,\C) \) then we can take \( X = Y -
  \lambda I \).  It is enough therefore to show that each \( X \in
  \gl(n,\C) \) with a zero eigenvalue may be obtained from a quiver
  of the desired form.  We prove this by induction on \( n \).

  Let \( X \) be such an element of \( \gl(n,\C) \); then put \( m =
  \rank X \) and choose an isomorphism \( \phi \colon {\im X}
  \rightarrow \C^m \).  Then taking \( \beta = \phi \circ X \colon
  \C^n \rightarrow \C^m \) and \( \alpha = \phi^{-1}\colon \C^m
  \rightarrow \im X \hookrightarrow \C^n \) we obtain \( X = \alpha
  \beta \) with \( \alpha \) injective and \( \beta \) surjective.

  Let \( Y = \beta \alpha \in \gl(m,\C) \) and pick an eigenvalue \(
  \mu \) of \( Y \).  By the inductive hypothesis there is a quiver
  diagram in some \( M(0<n_1<\dots<n_{r-1}=m) \) with all \( \alpha_i \)
  injective, all \( \beta_i \) surjective, satisfying the complex
  moment map equations for \( \prod_{i=1}^{r-2} \SL(n_i,\C) \) and
  such that \( \alpha_{r-2}\beta_{r-2} = Y - \mu I \).  This equation
  gives \( \alpha_{r-2} \beta_{r-2} - \beta \alpha= - \mu I \), so
  putting \( \alpha_{r-1} = \alpha \), \( \beta_{r-1} = \beta \) we
  obtain a quiver of the required form.
\end{proof}

\section{Stratification for quiver diagrams}
\label{sec:strat-quiv-diagr}

Let \( \mathbf n = ( 0 = n_0 \leqslant n_1 \leqslant n_2 \leqslant
\dots \leqslant n_r = n ) \) and consider the space \( M = M(\mathbf
n) \), defined at equation~\eqref{eq:Mn}, of hyperk\"ahler quiver
diagrams
\begin{equation}
  \label{eq:qd}
  0 \stackrel[\beta_0]{\alpha_0}{\rightleftarrows}
  \C^{n_1}\stackrel[\beta_1]{\alpha_1}{\rightleftarrows}
  \C^{n_2}\stackrel[\beta_2]{\alpha_2}{\rightleftarrows} \dots
  \stackrel[\beta_{r-2}]{\alpha_{r-2}}{\rightleftarrows} \C^{n_{r-1}}
  \stackrel[\beta_{r-1}]{\alpha_{r-1}}{\rightleftarrows} \C^{n_r} = \C^n.
\end{equation}
Given a relation \( S \) on \( \{1,\dots,r-1\} \), we wish to describe
the structure of \( Q_S = M \hkq H_S \), where \( H_S =
\varphi^{-1}(T_S) \) is the subgroup of \( \tH = \prod_{i=1}^{r-1}
\Un(n_i) \) specified in Definition\nobreakspace \ref {def:Hs}.

\begin{definition}
  \label{def:hkstable}
  A quiver diagram \( (\alpha,\beta) \in M \) will be called
  \emph{hyperk\"ahler stable} if, after applying some element of the
  group \( \SUr \) defined in~\eqref{eq:SU2-rot}, each \( \alpha_i \)
  is injective and each \( \beta_i \) is surjective.
\end{definition}

The hyperk\"ahler stable quivers form an open subset \( M^\hks \) of
\( M \), and, by Lemma\nobreakspace \ref {lem:SL-stabiliser}, \( H_S \) acts freely on the
intersection of \( M^\hks \) with the solution space of the
hyperk\"ahler moment map equations (since \( H_S \) is contained in
its complexification with respect to any complex structure, and the
definition of hyperk\"ahler stable means that for each quiver in \(
M^\hks \) we may choose a complex structure to which
Lemma\nobreakspace \ref {lem:SL-stabiliser} applies). It follows that the hyperk\"ahler
quotient \( M^\hks \hkq H_S \) is an open subset \( Q^\hks_S \) of
\( Q_S = M \hkq H_S \).

\begin{lemma}
  \label{lem:QHKS}
  \( Q^\hks_S \) is a hyperk\"ahler manifold.
\end{lemma}

\begin{proof}
  Since \( Q_S^\hks \) is an open subset of a hyperk\"ahler quotient
  by \( H_S \) it is enough to check that it is non-singular. This
  follows from the freeness of the \( H_S \) action on \( M^\hks \).
\end{proof}

\begin{remark}
  \label{rem:bigstar}
  The non-singularity result in Lemma\nobreakspace \ref {lem:QHKS} also follows from
  Lemmas\nobreakspace \ref {lem:SL-stabiliser} and\nobreakspace  \ref {lem:GITs} since (by the work of Kempf and
  Ness \cite{KN,Thomas}) for any choice of complex structures we can
  identify \( Q_S^\hks \) locally with the GIT quotient by \(
  (H_S)_\C \) of the affine subvariety of the affine space \( M \)
  defined by the complex moment map equations. Moreover the \(
  (H_S)_\C \) action is free on the neighbourhood where the
  identification takes place, and the subvariety is non-singular at
  any point whose stabiliser in \( (H_S)_\C \) is trivial (or even
  finite).
\end{remark}

\begin{lemma}
  \label{lem:hks}
  Suppose that, after applying some element of the group \( \SUr \) a
  quiver has the form \eqref{eq:qd} with each \( \alpha_i \)
  injective.  Then it is hyperk\"ahler stable.
\end{lemma}

\begin{proof}
  If a quiver \( (\alpha,\beta) \) has all \( \alpha_i \) injective,
  then the same is true for all but finitely many of the quivers in
  its orbit under the action of~\( \SUr \).  Now, the \( \SUr
  \)-action includes the right-multiplication by~\( \jj \) given
  in~\eqref{eq:j}.  We conclude that all but finitely elements of the
  \( \SUr \)-orbit have \( \beta^* \) injective, which is the same as
  saying that \( \beta \) is surjective.  Thus the quiver is
  hyperk\"ahler stable.
\end{proof}

For the stratification results, we will describe an augmentation
process for quiver diagrams.  Note that \( M(1,1) = \HH \) with \(
a+\jj b \) corresponding to the quiver \( \C
\stackrel[b]{a}{\rightleftarrows} \C \), whose maps are multiplication
by \( a \) and~\( b \), respectively.  Given any integer \( d > 0 \)
and indices \( 1 \leqslant i \leqslant j < r \), we have the element
\( e_{ij} = \sum_{k=i}^j e_k \in \Z^r \subset \R^r \), as in
Definition\nobreakspace \ref {def:Hs}, and may use the multiple \( d\,e_{ij} \) as a dimension
vector.  For \( p = j-i-1 \), we define a hyperk\"ahler embedding
\begin{gather*}
  \phi_{i,j;d}\colon \HH^p \longrightarrow M(d\,e_{ij})\\
  \phi_{i,j;d}(a+\jj b) = \bigl( 0 \stackrel[0]{0}{\rightleftarrows}
  \dots \stackrel[0]{0}{\rightleftarrows} 0
  \stackrel[0]{0}{\rightleftarrows} \C^d
  \stackrel[b_1]{a_1}{\rightleftarrows} \C^d
  \stackrel[b_2]{a_2}{\rightleftarrows} \dots
  \stackrel[b_p]{a_p}{\rightleftarrows} \C^d
  \stackrel[0]{0}{\rightleftarrows} 0 \dots
  \stackrel[0]{0}{\rightleftarrows} 0 \bigr)
\end{gather*}
with each map in the quiver being multiplication by the indicated
scalar (the sets on both sides are empty if \( p = -1 \)).  This map
is an isomorphism for \( d = 1 \), and in all cases is equivariant
with respect to the \( \SUr \)-action.  Our construction will involve
stacking the above embeddings on top of a quiver in \( M(\mathbf n)
\).

\begin{definition}
  A relation \( S \) on \( \{1,\dots,r-1\} \) is said to be
  \emph{injective} if \( S \) is the graph of an injective function,
  also denoted by \( S \), from \( \dom S = \{ i \mid \exists j: (i,j)
  \in S \} \) to \( \{1,\dots,r-1\} \).
\end{definition}

We note that any subrelation of an injective relation is again
injective.

\begin{definition}
  \label{def:augmentation}
  Let \( S \) be an injective subrelation of \( \leqslant \) on \(
  \{1,\dots,r-1\} \).  Suppose \( \mathbf m,\mathbf n \in
  \Z_{\geqslant0}^r \) are two ordered dimension vectors and there
  is a function \( \delta \colon {\dom S} \to \Z_{>0} \) such that
  \begin{equation*}
    \mathbf n = \mathbf m + \mathbf d,
  \end{equation*}
  for \( \mathbf d = \sum_{(i,j) \in S} \delta(i) e_{ij} \).  Put \( R
  = {<_S} = \{(i,j)\in S: i<j \} \) and \( \ell = \sum_{(i,j) \in R}
  j-i-1 \).

  The \emph{augmentation map} is the \( \SUr \)-equivariant
  hyperk\"ahler embedding
  \begin{equation*}
    \phi_{S,\delta} \colon M(\mathbf m) \oplus \HH^\ell \to M(\mathbf n)
  \end{equation*}
  obtained by writing \( \HH^\ell = \bigoplus_{(i,j) \in R}
  M(e_{ij}) \) and mapping
  \begin{equation*}
    ((\alpha,\beta),{(a^{(i)}+\jj b^{(i)})_{i \in \dom R}})
    \ \text{to}\ (\alpha,\beta) \oplus 0 \oplus \bigoplus 
    \phi_{i,S(i);\delta(i)}(a^{(i)}+\jj b^{(i)}),
  \end{equation*}
  where \( 0 \in M(\sum_{(i,i) \in S} d(i)e_i ) \).
\end{definition}

Recall now that the hyperk\"ahler
modification~\cite{Dancer-Swann:modifying} of a hyperk\"ahler
manifold~\( Y \) with a tri-Hamiltonian circle action is \(
Y_\textup{mod} = (Y \times \HH) \hkq S^1, \) where \( S^1 \) acts
diagonally on \( Y \times \HH \) with respect to the given action
on~\( Y \) and the inverse of the standard action on~\( \HH \).  This
construction adjusts to other choices of weight for the circle action
on~\( \HH \). More generally if a torus \( T^\ell \) acts on \( Y
\) with a hyperk\"ahler moment map then we have a hyperk\"ahler
modification
\begin{equation*}
  \hat Y = (Y \times \HH^\ell) \hkq T^\ell.
\end{equation*}
This is a hyperk\"ahler space of the same dimension as \( Y \) that
contains a copy of the hyperk\"ahler quotient \( Y\hkq T^\ell \).

In the situation of Definition\nobreakspace \ref {def:augmentation} above, suppose \( S
\subset S_1 \), where \( S_1 \) is also an injective subrelation of \(
\leqslant \).  Associated to \( S_1 \) we have the subgroup \( H_{S_1} =
H_{S_1}(\mathbf m) \) of \( \tH(\mathbf m) = \prod_{j=1}^{r-1}
\Un(m_j) \) from Definition\nobreakspace \ref {def:Hs}.  Put \( Q_1 = M(\mathbf m) \hkq
H_{S_1} \).  Then there is an action of \( \tH(\mathbf m) \) on \(
M(\mathbf m) \) and an induced action of \( \tT = \tH(\mathbf
m)/H(\mathbf m) \) on \( Q_1 \) since \( H_{S_1} \) contains \(
H(\mathbf m) \) as a normal subgroup.  

Embed \( T^\ell = \R^\ell/\Z^\ell \) into \( \tT \) via
\( e_{(i,j);k} \mapsto e_{i+k,j} \) for \( (i,j) \in S \) with \( j >
i \) and \( k=1,\dots,j-i \).  We may now use the induced action
of \( T^\ell \) on \( Q_1 \) to construct a hyperk\"ahler
modification
\begin{equation*}
  \hat Q_1 = (Q_1 \times \HH^\ell) \hkq T^\ell
\end{equation*}
and consider the open subset
\begin{equation}
  \label{eq:Q1hks}
  \hat Q_1^\hks = \bigl( Q_1^\hks \times (\HH\setminus\{0\})^\ell
  \bigr) \hkq T^\ell. 
\end{equation}

\begin{proposition}
  \label{prop:mod-embed}
  Suppose \( S_1 \) is an injective subrelation of~\( \leqslant \)
  on~\( \{1,\dots,\allowbreak{r-1\}} \) with \( S_1 = S \coprod S_2
  \), a disjoint union.  Let \( \mathbf m \), \( \mathbf n \), \( \delta
  \), \( R = {<_S} \) and~\( \ell \) be as in Definition\nobreakspace \ref {def:augmentation},
  with the additional assumption that \( m_r = n = n_r \), and write \(
  \phi_{S,\delta} \) for the resulting augmentation map.
  
  Put \( Q_1 = M(\mathbf m) \hkq H_{S_1} \), \( Q_2 = M(\mathbf n)
  \hkq H_{S_2} \) and let \( \hat Q_1^\hks \) be the open subset of
  the hyperk\"ahler modification ~\( \hat Q_1 \) as in
  equation~\eqref{eq:Q1hks}.  Then the augmentation map \( \phi_{S,\delta}
  \) induces an augmentation map
  \begin{equation*}
    \Phi_{S,\delta} \colon \hat Q_1 \to Q_2
  \end{equation*}
  of hyperk\"ahler quotients which is \( \SUr \)-equivariant and an
  embedding of the smooth manifold \( \hat Q_1^\hks \).
\end{proposition}

\begin{proof}
  A point of \( \hat Q_1 \) is represented by \( \mathbf q =
  ((\alpha,\beta),(a^{(i)}+\jj b^{(i)})_{i\in\dom R}) \) satisfying
  \begin{compactenum}
  \item the \( H_{S_1} \)-hyperk\"ahler moment map equations
    \begin{equation}
      \label{eq:HS1mu}
      \begin{gathered}
        \alpha_{i-1} \beta_{i-1} - \beta_i
        \alpha_i = \lambda^\C_i I, \\
        \alpha_{i-1}^{} \alpha_{i-1}^* - \beta_{i-1}^* \beta_{i-1}^{}
        + \beta_i^{} \beta_i^* - \alpha_i^* \alpha_i^{} = \lambda^\R_i
        I,
      \end{gathered}
    \end{equation}
    for \( i = 1,\dots,r-1 \), and for some \( (\lambda^\R,\lambda^\C)
    \) satisfying \( \sum_{k=i}^j \lambda^\C_k = 0 = \sum_{k=i}^j
    \lambda^\R_k \), for all \( (i,j) \in S_1 \),
  \item the \( T^{\strut\ell} \)-hyperk\"ahler moment map
    equations
    \begin{equation*}
      \begin{gathered}
        a_{k-i}^{(i)}b_{k-i}^{(i)} = \lambda^\C_k + \lambda^\C_{k+1}
        + \dots + \lambda^\C_j,\\
        \abs{a_{k-i}^{(i)}}^2 - \abs{b_{k-i}^{(i)}}^2 = \lambda^\R_k +
        \lambda^\R_{k+1} + \dots + \lambda^\R_j.
      \end{gathered}
    \end{equation*}
    for \( (i,j) \in R \) and \( i < k \leqslant j \).
  \end{compactenum}
  The quiver \( (\alpha',\beta') = \phi_{S,\delta}(\mathbf q) \) satisfies
  the analogue of the equations~\eqref{eq:HS1mu} with the same \(
  (\lambda^\R,\lambda^C) \).  In particular, since \( S_2 \subset S_1
  \), \( (\alpha',\beta') \) satisfies the \( H_{S_2} \)-hyperk\"ahler
  moment map equations.  Furthermore \( H_{S_1} \times T^\ell \)
  acts on \( (\alpha',\beta') \) as a subgroup of \( H_{S_2}
  \).  We thus have a well-defined, and \( \SUr
  \)-equivariant, map~\( \Phi_{S,\delta} \), as claimed.

  For \( \hat Q_1^\hks \), we can move our quiver via the \( \SUr \)
  action so each \( \alpha_i \) is injective, each \( \beta_i \) is
  surjective and all \( a^{(i)}_{k-i} \) and \( b^{(i)}_{k-i} \) are
  non-zero.  In this case, we can recover a point in \( \hat Q_1^\hks
  \) from the \( H_{S_2} \)-orbit of the quiver \( (\alpha',\beta') =
  \phi_{S,\delta}(\mathbf q) \in M(\mathbf n) \) by restricting maps
  in \( (\alpha',\beta') \) to the images of compositions of other
  maps in \( (\alpha',\beta') \), since the ambiguity in this process
  is exactly the action of \( H_{S_1} \times T^\ell \), cf.\
  Remark\nobreakspace \ref {rem:residual}.  Thus \( \Phi_{S,\delta} \) is injective on \(
  \hat Q_1^\hks \).
\end{proof}

\begin{remark}
  \label{rem:dot}
  We note that \( \hat Q_1^\hks \) determines \( Q_1^\hks \) as a
  hyperk\"ahler manifold.  Both spaces are the base of torus
  fibrations from a common total space contained in \( M(\mathbf m)
  \times \HH^\ell \).  The projection to \( \hat Q_1^\hks \) is a
  Riemannian submersion, whereas that to \( Q_1^\hks \) is just
  projection on to the \( M(\mathbf m) \) components.  The pull-backs
  of the two metrics differ by scalings along the quaternionic
  directions of the torus actions.  A description of this in terms of
  the twist construction may be found in~\cite{Swann:twist-mod}.
\end{remark}

\begin{remark}
  \label{rem:recover}
  The image \( \Phi_{S,\delta}(\hat Q_1) \subset Q_2 \) determines the
  original data \( \mathbf m \), \( S \) and~\( d \) as follows.
 
  Choose a point in the image represented by a quiver \(
  (\alpha',\beta') \) of largest possible rank.  Then \( m_i =
  \rank{\beta_i'\beta_{i+1}'\dots\beta_{r-1}'} \) and the
  corresponding quiver \( (\alpha,\beta) \) obtained by restriction
  has \( \alpha_i \) injective and \( \beta_i \) surjective.

  As \( S \) is injective the function \( \delta \) is determined by
  its values on the range of \( S \).  We determine \( S \) and \(
  \delta \) recursively.  Suppose we have found \( S' \subset S \) and
  the corresponding values of~\( \delta \).  Put \( \mathbf m' =
  \mathbf m + \sum_{(i,j)\in S'} \delta(i)e_{ij} \).  The largest
  element of the range of~\( S\setminus S' \) is the largest \( j \)
  for which \( m'_j < n_j \), the corresponding value of \( \delta \)
  is \(  n_j - m'_j \).  Now \( j = S(i) \), where \( i-1 < j \) is
  the largest index such that \(
  \rank{\beta_{i-1}'\beta_i'\dots\beta_{j-1}'} \) is strictly less
  than \( (n_j-m'_j) + \rank{\beta_{i-1}\beta_i\dots\beta_{j-1}} \).
\end{remark}

\begin{theorem}
  \label{thm:strat}
  Let \( \mathbf n = (n_1<\dots<n_r=n) \) be strictly ordered and let
  \( M(\mathbf n) \) be the space of hyperk\"ahler quiver
  diagrams~\eqref{eq:Mn}.  

  Suppose \( S \) is an injective subrelation of \( \leqslant \) on \(
  \{1,\dots,r-1\} \).  Let \( \delta\colon{\dom S} \to \Z_{>0} \) be a
  function such that \( \mathbf m = \mathbf n - \mathbf d \), \(
  \mathbf d = \sum_{(i,j)\in S}\delta(i)e_{ij} \), is an ordered dimension
  vector.  Put \( Q_S = M(\mathbf m)\hkq H_S \). Then
  \begin{equation*}
    Q_{(S,\delta)} = \Phi_{S,\delta}(\hat Q_S^\hks)
  \end{equation*}
  is a smooth hyperk\"ahler manifold that is a locally closed subset
  of~\( Q = M(\mathbf n) \hkq H \).

  Furthermore,
  \begin{equation*}
    Q = \coprod_{S,\delta} Q_{(S,\delta)} 
  \end{equation*}
  is the disjoint union over all such choices of \( S \) and \( \delta
  \).
\end{theorem}

\begin{proof}
  By Proposition\nobreakspace \ref {prop:mod-embed}, \( Q_{(S,\delta)} \) is a smooth
  hyperk\"ahler manifold.  It is open in its closure, which is just \(
  \Phi_{S,\delta}(\hat Q_S) \).  Remark\nobreakspace \ref {rem:recover} implies that \(
  Q_{(S,\delta)} \cap Q_{(S',\delta ')} = \varnothing \) if \(
  (S,\delta) \neq (S', \delta') \).  Finally it follows from
  Proposition\nobreakspace \ref {prop:gl-decomp}, that every quiver satisfying the
  hyperk\"ahler moment map equations for \( H \) lies in some \(
  Q_{(S,\delta)} \).
\end{proof}

\begin{remark}
  When \( S \) is empty, so that \( \delta \) is empty, we have \(
  Q_{(S,\delta)}= Q^\hks \).

  Let us now consider the full flag case when \( r = n \) and \( n_i =
  i \) for \( i \leqslant n \), so that \( Q = M \hkq H \) is the
  universal hyperk\"ahler implosion for \( \SU(n) \).  We specify \(
  (S,\delta) \), by listing the elements of \( S \), ordered by the
  first component, followed by the corresponding values of \( \delta
  \).

  When \( n = 2 \) there are two strata, \( Q_{(\varnothing,\varnothing)} =
  Q^\hks \) and \( Q_{(\{(1,1)\},1)} \) which consists of the zero
  quiver constructed as the direct sum of \( 0 \rightleftarrows 0
  \rightleftarrows \C^2 \) and \( 0 \rightleftarrows \C
  \rightleftarrows 0 \).

  When \( n = 3 \) the possible injective subrelations \( S \) of \(
  \leqslant \) on \( \{1,2\} \) are \( \varnothing \), the singletons
  \( \{(1,1)\} \), \( \{(2,2)\} \), \( \{(1,2)\} \), and the subset \(
  \{(1,1), (2,2)\} \).  The strata are as follows:
  \begin{asparaenum}
  \item \( Q_{(\varnothing,\varnothing)} = Q^\hks \);
  \item \( Q_{(\{(1,1)\},1)} \) with elements given by the direct sum
    of a hyperk\"ahler stable quiver \( 0 \rightleftarrows 0
    \rightleftarrows \C^2 \rightleftarrows \C^3 \) and the zero quiver
    \( 0 \rightleftarrows \C \rightleftarrows 0 \rightleftarrows 0 \);
  \item \( Q_{(\{(2,2)\},1)} \) with elements given by the direct sum
    of a hyperk\"ahler stable quiver \( 0 \rightleftarrows \C
    \rightleftarrows \C \rightleftarrows \C^3 \) and the zero quiver
    \( 0 \rightleftarrows 0 \rightleftarrows \C \rightleftarrows 0 \);
  \item \( Q_{(\{(1,2)\},1)} \) with elements given by the direct sum
    of a hyperk\"ahler stable quiver \( 0 \rightleftarrows 0
    \rightleftarrows \C \rightleftarrows \C^3 \) and a quiver of the
    form \( 0 \rightleftarrows \C \rightleftarrows \C \rightleftarrows
    0 \) where the maps \( \C \rightleftarrows \C \) are isomorphisms;
  \item \( Q_{(\{(1,1),(2,2)\},(1,1))} \) with elements given by the
    direct sum of a hyperk\"ahler stable quiver \( 0 \rightleftarrows
    0 \rightleftarrows \C \rightleftarrows \C^3 \) and the zero
    quivers \( 0 \rightleftarrows \C \rightleftarrows 0
    \rightleftarrows 0 \) and \( 0 \rightleftarrows 0 \rightleftarrows
    \C \rightleftarrows 0 \);
  \item \( Q_{(\{(1,1),(2,2)\},(1,2))} \) which consists of the zero
    quiver constructed as the direct sum of \( 0 \rightleftarrows 0
    \rightleftarrows 0 \rightleftarrows \C^3 \) and the zero quivers
    \( 0 \rightleftarrows \C \rightleftarrows 0 \rightleftarrows 0 \)
    and \( 0 \rightleftarrows 0 \rightleftarrows \C^2 \rightleftarrows
    0 \).
  \end{asparaenum}
\end{remark}

\begin{corollary}
  \label{cor:dense}
  \( Q^\hks \) is a dense open subset of \( Q \) with complement of
  complex codimension at least \( 2 \).
\end{corollary}

\begin{proof}
  We observe using the complex equations that \( \rank (\alpha_{i-1}
  \beta_{i-1} - \lambda_i^\C I) = \rank (\beta_i \alpha_i) < \min (\rank
  \alpha_i, \rank \beta_i) \), so if \( \alpha_i \) or \( \beta_i \)
  has rank less than \( i \) then \( \lambda_i^\C \) is an eigenvalue of
  \( \alpha_{i-1} \beta_{i-1} \).

  Now, as the map \( (\alpha, \beta) \mapsto \beta \alpha \) is a
  surjection from \( \Hom(\C^j,\C^{j+1}) \oplus \Hom(\C^{j+1},\C^j)
  \) onto \( \Hom (\C^j,\C^j) \) for all \( j \) we may vary the
  \( \lambda_i^\C \) arbitrarily at each stage (without changing \(
  \alpha_j, \beta_j \) for \( j < i \)), and still stay within \(
  \mu_\C^{-1}(0) \). In particular the \( H_\C \)-invariant function
  \( \Phi_i \coloneqq \det(\alpha_{i-1} \beta_{i-1} - \lambda_i^\C I) \)
  cannot vanish identically on a non-empty open set in \( \mu_\C=0 \).
  The complement of \( \Phi_i^{-1}(0) \) is now open and dense in the
  locus where \( \mu_\C=0 \) for each \( i \).

  This implies that \( Q^\hks \) is dense in \( Q \); the complement
  is a union of strata of even complex dimension, hence the
  codimension statement follows.
\end{proof}

\section{The structure of the strata}
\label{sec:structure-strata}

We shall now take a closer look at the structure of the strata \(
Q_{(S,\delta)} \) appearing in the stratification of the quiver space
\( Q = M \hkq H \) with dimension vector \( (n_1,\dots,n_r = n) \)
given by Theorem\nobreakspace \ref {thm:strat}. Recall that \( Q_{(S,\delta)} \) is an open subset of 
a hyperk\"ahler modification of a quotient of the
form
\begin{equation*}
  M^\hks \hkq H_S = (M^\hks \hkq H) \hkq T_S = Q^\hks \hkq T_S 
\end{equation*}
for a different dimension vector \( (m_1,\dots,m_r) \) and a subtorus
\( T_S \) of \( \tT = T^{r-1} \). Thus we shall first study
the open stratum \( Q^\hks = Q_{(\varnothing, \varnothing)} = M^\hks
\hkq H\) consisting of the hyperk\"ahler stable quivers.

\bigbreak We first look at quivers
\begin{equation*}
  0 \stackrel[\beta_0]{\alpha_0}{\rightleftarrows}
  \C^{n_1}\stackrel[\beta_1]{\alpha_1}{\rightleftarrows}
  \C^{n_2}\stackrel[\beta_2]{\alpha_2}{\rightleftarrows}\dots
  \stackrel[\beta_{r-2}]{\alpha_{r-2}}{\rightleftarrows} \C^{n_{r-1}}
  \stackrel[\beta_{r-1}]{\alpha_{r-1}}{\rightleftarrows} \C^{n_r} = \C^n
\end{equation*}
satisfying the complex moment map equations
\begin{equation*}
  \alpha_{i-1} \beta_{i-1} = \beta_i \alpha_i + \lambda_i^\C I_{n_i}
\end{equation*}
for \( H \) where \( 0 = n_0 \leqslant n_1 \leqslant \dots \leqslant
n_r = n \) and all the \( \beta_i \) are surjective. Lemma\nobreakspace \ref {lem:GITs}
shows that the \( H_\C \)-orbits of all such quivers are closed and
hence represent points in \( Q^\hks \) by Lemma\nobreakspace \ref {lem:hks}.

As in the proof of Lemma\nobreakspace \ref {lem:sympstrata} we may choose bases for the
vector spaces \( V_i = \C^{n_i} \) so that
\begin{equation*}
  \beta_i = \left( 0_{n_i \times k_i} \mid I_{n_i \times n_i} \right)
\end{equation*}
where \( k_i = n_{i+1} - n_i \) is the dimension of the kernel of \(
\beta_i \).  This amounts to using the action of \( \tH_\C \times
\SL(n,\C) = \prod_{i=1}^r \GL(n_i,\C) \) to standardise the \( \beta_i
\), and we can replace \( \GL(n_i,\C) \) with \( \SL(n_i,\C) \) for
each \( i \) such that \( n_{i-1} < n_i \), so that if the dimension
vector is strictly ordered then it amounts to using the action of \(
H_\C \times \SL(n,\C) = \prod_{i=1}^r \SL(n_i,\C) \).

Let us now assume we are in the strictly ordered case.  As we saw in
Lemma\nobreakspace \ref {lem:sympstrata}, when the dimension vector is strictly ordered the
remaining freedom in the group action is the commutator of the
parabolic group \( P \) in \( \SL(n,\C) \) associated to the flag of
dimensions \( (n_1, n_2, \dots, n_r = n) \) in \( \C^n \).  In the
particular case when \( n_i =i \) for all \( i \) (that is, all \( k_i
\) equal 1), this freedom is exactly the maximal unipotent group \( N
\); that is, the commutator subgroup of the Borel group \( B \).

Now let us investigate what \( X =\alpha_{r-1} \beta_{r-1} \) tells us
when the dimension vector is strictly ordered once the \( \beta_i \)
have been standardised as above.  With respect to bases chosen as
above, the matrix of \( \alpha_i \beta_i \) is
\begin{equation}
  \label{eq:abform}
  \begin{pmatrix}
    0_{k_i \times k_i} & D_{k_i \times n_i} \\
    0_{n_i \times k_i} & -\lambda_i^\C I_{n_i} + \alpha_{i-1}
    \beta_{i-1}
  \end{pmatrix}
\end{equation}
for some \( D \).

Inductively it is now easy to show that \( \alpha_i \beta_i \) has
scalar blocks of size \( k_j \times k_j \) (\( j=i, i-1, \dots, 0 \))
down the diagonal, where the scalars (from top left going down) are \(
0, -\lambda_i^\C, -(\lambda_i^\C + \lambda_{i-1}^\C), \dots,
-(\lambda_i^\C + \dots + \lambda_1^\C) \).

In particular \( X = \alpha_{r-1} \beta_{r-1} \) lies in the
annihilator of the Lie algebra of the commutator \( [P,P] \) of the
parabolic determined by the integers \( k_j \).  Again, in the case \(
n_i=i \) we have that \( X \) lies in the Borel subalgebra \( \bmf =
\n^\circ \). Notice also that the diagonal entries of \( X \) are \( 0
\) \( (k_{r-1} \) times), \( -\lambda_{r-1}^\C \) \( (k_{r-2} \)
times), \( \dots, -(\lambda_{r-1}^\C + \dots + \lambda_1^\C) \) \(
(k_0=n_1 \) times).
 
Moreover any such \( X \) comes from a solution to our equations,
because we have that \( X \) kills \( \ker \beta_{r-1} \), and \(
\lambda_i^\C + \beta_i \alpha_i \) kills \( \ker \beta_{i-1} \) for \(
i < r-1 \).

Observe that for each \( i \), knowledge of \( \alpha_i \beta_i \)
determines \( \alpha_i \) (since \( \beta_i \) is surjective and
standardised), hence determines \( \beta_i \alpha_i \) (since \(
\beta_i \) is standardised), and hence, together with knowledge of \(
\lambda_i^\C \), determines \( \alpha_{i-1} \beta_{i-1} \) by the
equations.

As we can read off the \( \lambda_i^\C \) by looking at the diagonal
entries of \( X \), we see that knowledge of \( X \) determines all
the \( \alpha_i \) and hence the whole diagram.

In addition \( X \) is determined by its trace-free part, as its
leading entry \( X_{11}=0 \).  So, in summary, we have shown, in the
strictly ordered case, that if the \( \beta_i \) are surjective they
may be standardised (modulo \( H_\C \)) by an element of \( \SL(n,\C)
\), unique up to an element of the commutator of the parabolic
subgroup, and now the whole diagram is determined by \( X \) (or its
trace-free part) in the annihilator of the Lie algebra of this
commutator. Moreover any such \( X \) arises from such a diagram. We
summarise our results as follows.

\begin{proposition}
  \label{prop:betasurj}
  Consider a quiver diagram with strictly ordered dimension vector \(
  (n_1, \dots, n_r=n) \). Then the set of solutions to the complex
  moment map equations for \( H \) with \( \beta_i \) surjective,
  modulo the action of \( H_\C=\prod_{i=1}^{r-1} \SL(n_i,\C) \), may
  be identified with
  \begin{equation*}
    \SL(n,\C) \times_{[P,P]} [\p, \p]^\circ
  \end{equation*}
  where \( P \) is the parabolic subgroup associated to the flag \(
  (n_1, \dots, n_r=n) \), and \( [\p, \p]^\circ \) is the annihilator
  of the Lie algebra of the commutator subgroup of \( P \).

  In the special (full flag) case where \( n_i =i \) for all \( i \),
  we obtain the space
  \begin{equation*}
    \SL(n,\C) \times_N \bmf
  \end{equation*}
  where \( N \) is a maximal unipotent subgroup of \( \SL(n,\C) \) and
  \( \bmf = \n^\circ \) is a Borel subalgebra.
\end{proposition}

\begin{remark}
  The space \( \SL(n,\C) \times_N \bmf \) has also occurred in work
  of Bielawski \cite{Bielawski:hyper-kaehler,Bielawski:GM}.
\end{remark}

We also have:

\begin{proposition}
  \label{prop:betasurj2}
  Consider a quiver diagram with strictly ordered dimension vector \(
  (n_1, \dots, n_r=n) \). Then the set of solutions to the complex
  moment map equations for \( H \) with \( \alpha_i \) injective and
  \( \beta_i \) surjective, modulo the action of \( H_\C =
  \prod_{i=1}^{r-1} \SL(n_i,\C) \), may be identified with
  \begin{equation*}
    \SL(n,\C) \times_{[P,P]} [\p, \p]^\circ_*
  \end{equation*}
  where \( P \) is the parabolic subgroup associated to the flag \(
  (n_1, \dots, n_r=n) \), and \( [\p, \p]^\circ_* \) is an open dense
  subset of \( [\p,\p]^\circ \). Moreover \( [\p, \p]^\circ_* \) is
  contained in the complement of the union over all parabolic
  subgroups \( P^{\prime} \) strictly containing \( P \) of the
  annihilator \( [\p^{\prime},\p^{\prime}]^\circ \) of the Lie algebra
  of the commutator subgroup of \( P^{\prime} \).

  In the full flag case where \( n_i =i \) for all \( i \), we obtain
  a space
  \begin{equation*}
    \SL(n,\C) \times_N \bmf_*
  \end{equation*}
  where \( N \) is a maximal unipotent subgroup of \( \SL(n,\C) \) and
  \( \bmf = \n^\circ \) is a Borel subalgebra.
\end{proposition}

\begin{proof}
  The statement that \( X \) has to lie in the complement of \(
  [\p^{\prime},\p^{\prime}]^\circ \) follows by an induction, using
  \eqref{eq:abform} and the given form of \( \beta_i \).
\end{proof}

How does the argument which gave us Propositions\nobreakspace \ref {prop:betasurj} and\nobreakspace  \ref {prop:betasurj2} need to
be modified if the dimension vector \( (n_1, \dots, n_r=n) \) is
ordered but not strictly ordered? We may still choose bases for the
vector spaces \( V_i = \C^{n_i} \) so that
\begin{equation*}
  \beta_i = \left( 0_{n_i \times k_i} \mid I_{n_i \times n_i} \right)
\end{equation*}
where \( k_i = n_{i+1} - n_i \) is the dimension of the kernel of \(
\beta_i \), but to do this we may need to use the action of a larger
group than \( H_\C \times \SL(n,\C) = \prod_{i=1}^r \SL(n_i,\C) \).
It suffices to use the action of \( \tH_\C \times \SL(n,\C) =
\bigl(\prod_{i=1}^{r-1}\GL(n_i,\C)\bigr) \times \SL(n_r,\C) \), and
then the remaining freedom is \( P \) itself, embedded in \( \tH_\C
\times \SL(n,\C) \) so that the projection of \( g \in P \) in \(
\GL(n_i,\C) \) is the bottom right hand \( n_i \times n_i \) block of
\( g \). Let
\begin{equation*}
  \tT_\C = \tH_\C /H_\C =  \prod_{i=1}^{r-1}
  \GL(n_i,\C)/\SL(n_i,\C) =  (\C^*)^{r-1}.  
\end{equation*}
Once we have quotiented by the action of \( H_\C \) we are using the
residual action of \( \tT_\C \times \SL(n,\C) \) to put the maps
\( \beta_i \) into standard form, and the remaining freedom is the
action of \( P \) embedded in \( \tT_\C \times \SL(n,\C) \) via
the inclusion in \( \SL(n,\C) \) and the homomorphism
\begin{equation}
  \label{eq:chi}
  \chi = (\chi_1, \dots, \chi_{r-1}) \colon P  \to \tT_\C =
  (\C^*)^{r-1}, 
\end{equation}
where \( \chi_i\colon P \to \C^* \) is the character given by the
determinant of the bottom right hand block of size \( n_i \times n_i
\).  Note that the kernel of \( \chi \) is the commutator subgroup \(
[P,P] \) of \( P \), but that \( \chi \) is not surjective if \( k_i =
0 \) (i.e.\ if \( n_{i+1} = n_i \)) for some \( i \).

The complex moment map equations
\begin{equation*}
  \alpha_{i-1} \beta_{i-1} = \beta_i \alpha_i + \lambda_i^\C I_{n_i}
\end{equation*}
tell us that
\begin{equation*}
  \alpha_i =
  \begin{pmatrix}
    \alpha_i^{(1)} \\
    \alpha_i^{(2)}
  \end{pmatrix}
  , 
\end{equation*}
where \( \alpha_i^{(1)} \) is \( k_i \times n_i \) and \(
\alpha_i^{(2)} \) is \( n_i \times n_i \) and
\begin{equation*}
  \alpha_{i+1}^{(2)} =  
  \begin{pmatrix}
    - \lambda^\C_{i+1}I_{k_i \times k_i} & \alpha_i^{(1)} \\
    0_{n_i \times k_i} & \alpha_i^{(2)} - \lambda^\C_{i+1}I_{n_i \times n_i}
  \end{pmatrix}
  . 
\end{equation*}
Inductively it follows that \( \alpha_i^{(2)} \) has scalar blocks of
size \( k_j \times k_j \) (for \( j= r-1, \dots, 0 \)) down the diagonal,
where the scalars (from top left going down) are 
\begin{equation*}
  -\lambda_i^\C,
  -(\lambda_i^\C + \lambda_{i-1}^\C), \dots, -(\lambda_i^\C + \dots +
  \lambda_1^\C). 
\end{equation*}
Furthermore the quiver is determined by knowledge of \( \lambda_1^\C,
\dots, \lambda_{r-1}^\C \) together with
\begin{equation*}
  X = \alpha_{r-1} \beta_{r-1} =
  \begin{pmatrix}
    0_{k_{r-1} \times k_{r-1}} & \alpha_{r-1}^{(1)} \\
    0_{n_{r-1} \times k_{r-1}} & \alpha_{r-1}^{(2)}
  \end{pmatrix}
\end{equation*} 
which is block triangular with scalar blocks of size \( k_j \times k_j
\) (\( j=  r-1, \dots, 0 \)) down the diagonal,  where the scalars are
\(  0,   -\lambda_{r-1}^\C,  -(\lambda_{r-1}^\C  +  \lambda_{r-2}^\C),
\dots, -(\lambda_{r-1}^\C  + \dots +  \lambda_1^\C) \).  Note  that we
can only  recover from  \( X  \) those \(  \lambda_{r-1}^\C +  \dots +
\lambda_i^\C  \)  for  which  \(  k_{i-1}  > 0  \);  however  from  \(
\lambda_1^\C,  \dots, \lambda_{r-1}^\C  \)  (which we  can regard  as
determining  an  element  of  the  dual  of  the  Lie  algebra  of  \(
\tT_\C \)) and the off-diagonal blocks  of \( X \) (that is, the
projection of \(  X \) to \( \p^\circ \)) we  can recover the quiver
with \( \beta_i \) in standardised form.

\begin{remark}
  Now let \( S \) be an injective subrelation of \( \leqslant \) on \(
  \{1,\dots,r-1\} \).  Let \( T_S \) be the subtorus of \( \tT =
  T^{r-1} \) as in Definition\nobreakspace \ref {def:Hs} and let \( H_S = \phi^{-1} (T_S)
  \) be the corresponding subgroup of~\( \tH \) containing~\( H \).
  Then the complex moment map equations for \( H_S \) are given by the
  complex moment map equations
  \begin{equation*}
    \alpha_{i-1} \beta_{i-1} = \beta_i \alpha_i + \lambda_i^\C I_{n_i}
  \end{equation*}
  for \( H \) together with the equations
  \begin{equation*}
    \lambda^\C_i + \lambda^\C_{i+1} + \dots + \lambda^\C_j = 0
    \qquad\text{for \( (i,j) \in S \),}
  \end{equation*}
  which say that \( (\lambda_1^\C,\dots, \lambda_{r-1}^\C) \in
  \LIE(\tT_\C)^* \) lies in the annihilator of \( \LIE({T_S})_\C
  \).  Furthermore the residual action of \( (H_S)_\C/H_\C = (T_S)_\C
  \) is given by its embedding as a subgroup of \( \tT_\C \) and
  thus is a subgroup of \( \tT_\C \times \SL(n,\C) \).
\end{remark}

Putting all this together gives us the following extension of
Proposition\nobreakspace \ref {prop:betasurj2}.  It is a generalisation, since when the
dimension vector is strictly ordered then the homomorphism \(
\chi\colon P \to \tT_\C \) defined at \eqref{eq:chi} above is
surjective with kernel \( [P,P] \), so that
\begin{equation*}
  (\tT_\C \times \SL(n,\C)) / P \cong \SL(n,\C)/[P,P].
\end{equation*}

\begin{proposition}
  \label{prop:betasurjgen2}
  Consider a quiver diagram with ordered dimension vector \( (n_1,
  \dots, n_r=n) \). Then the set of solutions to the complex moment
  map equations for \( H \) with \( \alpha_i \) injective and \(
  \beta_i \) surjective, modulo the action of \(
  H_\C=\prod_{i=1}^{r-1} \SL(n_i,\C) \), may be identified with an
  open subset of the cotangent bundle to
  \begin{equation*}
    (\tT_\C \times \SL(n,\C)) / P,
  \end{equation*} or equivalently with
  \begin{equation*}
    (\tT_\C \times \SL(n,\C)) \times_P (\LIE(\tT_\C)^*
    \oplus  \p^\circ)^* 
  \end{equation*}
  where \( P \) is the parabolic subgroup associated to the flag \(
  (n_1, \dots, n_r=n) \). Here
  \begin{equation*}
    \tT_\C = \tH_\C/H_\C = (\C^*)^{r-1}
  \end{equation*}
  with \( P \) acting on \( \SL(n,\C) \) by left multiplication and on
  \( \tT_\C \) via the characters given by the determinants of
  the bottom right hand blocks of size \( n_i \times n_i \), and \(
  (\LIE(\tT_\C)^* \oplus \p^\circ)^* \) is an open dense
  subset of \( \LIE(\tT_\C)^* \oplus \p^\circ \).  Moreover
  the set of solutions to the complex moment map equations for \(
  H_S=\phi^{-1}(T_S) \) (defined as in Definition\nobreakspace \ref {def:Hs}) with \( \alpha_i \)
  injective and \( \beta_i \) surjective, modulo the action of \(
  (H_S)_\C \), may be identified with an open subset of the cotangent
  bundle to
  \begin{equation*}
    \tT_\C \times \SL(n,\C) / {(T_S)_\C \times P}
  \end{equation*}
  where \( (T_S)_\C = (H_S)_\C/H_\C \) is a subgroup of \( \tT_\C \) and
\( (T_S)_\C \times P \) is embedded as a subgroup of \( \tT_\C \times
  \SL(n,\C) \) via
\begin{equation*}
(t,g) \mapsto (t\chi(g), g)
\end{equation*}
with \( \chi \colon P \to \tT_\C \) as defined at \eqref{eq:chi}.
\end{proposition}

\begin{remark}
  An alternative argument notes that since the \( \alpha_i \) are
  injective and the \( \beta_i \) are surjective then \( \alpha_i \)
  and \( \beta_i \) are isomorphisms whenever \( n_i= n_{i+1} \) and
  so we can contract the quiver as in Proposition\nobreakspace \ref {prop:gl-decomp} until we obtain
  a quiver of the same form but with strictly ordered dimension
  vector.  For each such contraction the information lost is \(
  \lambda_i^\C \in \C \) and the difference between the actions of \(
  \GL(V_{i+1}) \) and \( \SL(V_{i+1}) \) is \(
  \GL(V_{i+1})/{\SL(V_{i+1})} \cong \C^* \), so we can use this point of
  view to deduce Proposition\nobreakspace \ref {prop:betasurjgen2} from Proposition\nobreakspace \ref {prop:betasurj2}.
\end{remark}

\begin{remark}
  \label{rem:lambdaconstr}
  In fact for certain values of \( \lambda^\C \) we get restrictions
  on which parabolics can occur with a non-empty solution set in
  Proposition\nobreakspace \ref {prop:betasurjgen2}. More precisely, we observe that if \(
  \lambda_{i+1}^\C=0 \) then the complex moment map equations imply \(
  \alpha_{i+1} \) maps \( \ker \beta_i \) into \( \ker \beta_{i+1} \).
  If all \( \beta_j \) are surjective and \( \alpha_j \) are
  injective, this means that \( k_i \leqslant k_{i+1} \).  In
  particular if all \( \lambda_i^\C \) are zero, then the \( k_i \)
  form a \emph{non-decreasing} partition of \( n \). This is to be
  expected, as such partitions count the number of \emph{unordered}
  partitions, that is, the strata of the nilpotent variety.
\end{remark}

\begin{remark}
  Now suppose that we are in the situation of Theorem\nobreakspace \ref {thm:strat}.
  Recall that then the stratum \( Q_{(S,\delta)} \) is the image of
  the hyperk\"ahler embedding into \( Q = M \hkq H \) defined in
  Proposition\nobreakspace \ref {prop:mod-embed} with \( S_1 = S \) and \( S_2 = \varnothing \)
  of the hyperk\"ahler modification \( \hat Q_1^\hks \) of \( Q_1^\hks \)
  as in~\eqref{eq:Q1hks}.

  Observe that \( \hat Q_1^\hks = \hat M(\mathbf m)^\hks \hkq H_S \), where
  \begin{equation}
    \label{eq:hMell}
    \hat M(\mathbf m)^\hks = (M(\mathbf m)^\hks \times
    (\HH\setminus\{0\})^\ell) \hkq T^\ell.
  \end{equation}
  Let \( h = h_M-h_\ell \) be the hyperk\"ahler moment map used
  in~\eqref{eq:hMell}; this takes values in \( \R^\ell \oplus \C^\ell
  \).  We have that \( \hat M(\mathbf m)^\hks = (h^\C)^{-1}(0) \sslash
  T_\C^\ell \).  Put \( M_0 = (h_M^\C)^{-1}((\C\setminus\{0\})^\ell)
  \).  Now the \( i \)th component of \( h_\ell^\C \) is just \(
  a^{(i)}+\jj b^{(i)} \mapsto a^{(i)}b^{(i)} \), so for \( m\in M_0
  \), \( h_M^\C(m) = h_\ell^\C((a^{(i)}+\jj b^{(i)})) \) implies that
  each \( a^{(i)} \) and \( b^{(i)} \) is non-zero, and hence the \(
  T_\C^\ell \) orbit through the point is closed.  In particular, the
  holomorphic map \( m \mapsto (m,h_M^\C(m) - \jj\mathbf 1) \), \(
  \mathbf 1 = (1,\dots,1) \in \R^\ell \), realises \( M_0 \) as an
  open subset of \( h_\C^{-1}(0)\sslash T_\C^\ell = \hat M(\mathbf
  m)^\hks \).  This map is equivariant for both \( H_S^\C \) and \(
  \SL(n,\C) \), and so descends to a holomorphic map on an open dense
  set of \( Q_1 \) to~\( \hat Q_1^\hks\).  Exploiting the action of \(
  \SUr \), we can cover \( \hat Q_1^\hks \) by such open sets.

  Now note that if \( m_{j+1} = m_j \) for some \( j \), then, since
  \( n_{j+1} > n_j \) there exists some \( i \leqslant j \) such that
  \( (i,j) \in S \). This implies that the homomorphism
  \begin{equation*}
    (T_S)_\C \times P \to \tT_\C
  \end{equation*}
  given by \( (t,g) \mapsto t\chi(g) \) is surjective, and its kernel
  is isomorphic to the subgroup
  \begin{equation*}
    P_S = \{ g \in P:\chi(g) \in (T_S)_\C \}
  \end{equation*}
  of \( P \) containing \( [P,P] \). Thus
  \begin{equation*}
    \tT_\C \times \SL(n,\C) /(T_S)_\C \times P
  \end{equation*}
  can be identified with \( \SL(n,\C)/P_S \), and its cotangent bundle
  can be identified with
  \begin{equation*}
    \SL(n,\C) \times_{P_S} \p_S^\circ
  \end{equation*}
  where \( \p^\circ \) is the annihilator in the dual of the Lie
  algebra of \( \SL(n,\C) \) of the Lie algebra \( \p_S \) of \( P_S
  \).
\end{remark}

Thus we have

\begin{theorem}
  \label{thm:str-struct}
  In the situation of Theorem\nobreakspace \ref {thm:strat} each stratum \( Q_{(S,\delta)}
  \) of \( Q \) is a union of open subsets, one for each element of \(
  \SUr \), each of which can be identified with
  \begin{equation*}
    \SL(n,\C) \times_{P_S} (\p_S)^\circ_*.
  \end{equation*}
  Here \( (\p_S)^\circ_* \) is an open subset of the annihilator \(
  (\p_S)^\circ \) in \( \mathrm{Lie}(\SL(n,\C))^* \) of the Lie
  algebra \( \p_S \) of a subgroup \( P_S \) of the standard parabolic
  subgroup \( P \) of \( \SL(n,\C) \) associated to the flag \(
  \C^{m_1} \leqslant \dots \leqslant \C^{m_r}=\C^n \) where
  \begin{equation*}
    m_k = k - d_k,
  \end{equation*}
  \( d_k \) the \( k \)th component of \( \mathbf d = \sum_{(i,j)\in
  S} \delta(i)e_{ij} \).  More precisely, \( P_S = \{ g \in P:\chi(g)
  \in (T_S)_\C \} \) where \( (T_S)_\C = (H_S)_\C/H_\C \) is the
  subgroup of \( \tT_\C = \tH_\C/H_\C \) defined as in
  Definition\nobreakspace \ref {def:Hs} and \( \chi \colon P \to \tT_\C \) is defined
  at~\eqref{eq:chi}.
\end{theorem}

\begin{corollary}
  \label{cor:full-struct}
  In the case of a full flag, when \( r=n \) and \( n_j = j \) for \(
  j=0,\dots , n \), this description applies to each stratum \(
  Q_{(S,\delta)} \) of the stratification described in Theorem\nobreakspace \ref {thm:strat}
  of the universal hyperk\"ahler implosion \( Q = M \hkq H \) for \(
  K=\SU(n) \).
\end{corollary}

\begin{remark}
  \label{rem:cs-strata}
  \( \SL(n,\C) \times_{P_S} \p_S^\circ \) can be regarded as the
  complex-symplectic quotient of \( T^*K_\C=K_\C \times \kf_\C \) by
  \( P_S \), where \( K = \SU(n) \) and \( [P,P] \leqslant P_S
  \leqslant P \).  We may compare this, as in
  \S \ref {sec:towards-hyperk-impl}, with symplectic implosion, where the
  universal implosion \( (T^*K)_\impl = K_\C \symp N \) is stratified
  by the ordinary quotients of \( K_\C \) by commutators of
  parabolics.
\end{remark}

\begin{remark}
  \label{rem:Feix}
  We have noted that for general compact \( K \), the space \( K_\C
  \times_{[P,P]} [\p,\p]^\circ \) may be viewed as the cotangent
  bundle of \( K_\C/[P,P] \). Now, from~\cite{Guillemin-JS:implosion}
  the latter quotient is just a K\"ahler stratum of the symplectic
  implosion \( K_\C \symp N \) (cf.~Remark\nobreakspace \ref {rem:cs-strata}).  A theorem of
  Feix~\cite{Feix}, now shows that there is a hyperk\"ahler metric on
  some open neighbourhood of the zero section in the cotangent bundle
  \( K_\C \times_{[P,P]} [\p,\p]^\circ \).  Proposition\nobreakspace \ref {prop:betasurj} gives us a
  hyperk\"ahler structure on the full set \( K_\C \times_{[P,P]}
  [\p,\p]^\circ \) for \( K = \SU(n) \)
\end{remark}

\begin{remark}
  \label{rem:ident}
  Note that in the full flag case the homomorphism \( \chi\colon B \to
  \tT_\C \) defined at~\eqref{eq:chi} is surjective with
  kernel \( N=[B,B] \) and thus allows us to identify \( \tT_\C
  \) naturally with the maximal torus \( T_\C \) of \( K_\C \).
\end{remark}

\bigbreak We would also like to relate the quiver space \( Q = M \hkq
H \) in the full flag case to the non-reductive GIT quotient \(
(\SL(n,\C) \times \bmf) \symp N \), which as discussed in
\S \ref {sec:towards-hyperk-impl} could be interpreted as a
complex-symplectic quotient in the GIT sense of the cotangent bundle
\( T^* \SL(n,\C) \) by the maximal unipotent \( N \).

\begin{lemma}
  \label{lem:codtwo}
  When \( r=n \) and \( n_j = j \) for \( j=0,\dots , n \) the
  complement in the variety defined by the complex moment map
  equations \( \mu_\C=0 \) of the locus of full flag quivers with all
  \( \alpha_i \) injective and \( \beta_i \) surjective has complex
  codimension at least \( 2 \).
\end{lemma}

\begin{proof}
  As in the proof of Corollary\nobreakspace \ref {cor:dense}, we observe that if some \(
  \alpha_i \) or \( \beta_i \) is of less than maximal rank, then the
  \( H_\C \)-invariant function \( \Phi_i \) given by \(
  \det(\alpha_{i-1} \beta_{i-1} - \lambda_i^\C I) \) is zero; moreover
  we may vary the \( \lambda_i^\C \) arbitrarily at each stage
  (without changing \( \alpha_j, \beta_j \) for \( j < i \)), and
  still stay within \( \mu_\C^{-1}(0) \) since the map \( (\alpha,
  \beta) \mapsto \beta \alpha \) is a surjection from \( \Hom(\C^j,
  \C^{j+1}) \oplus \Hom(\C^{j+1},\C^j) \) onto \( \Hom (\C^j,
  \C^j) \) for all \( j \). In particular the the zero locus of \(
  \Phi_i \) is a variety of codimension one. Furthermore, if this
  occurs for two indices \( i \) and \( j \), we can see in the same
  way that both \( \Phi_i \) and \( \Phi_j \) vanish, and this is a
  codimension two condition because \( \Phi_j \) does not vanish
  identically on a non-empty open set in the zero locus of \( \Phi_i
  \).

  Similarly, if for some \( i \) the rank of \( \alpha_i \) or \(
  \beta_i \) is less than \( i-1 \), we deduce that \( \Phi_i \)
  vanishes to order at least two, so we are in codimension two or
  higher.

  So we just have to consider the situation where there is only one
  index \( i \) where \( \alpha_i \) or \( \beta_i \) are of less than
  full rank, and for this index \( (\rank \alpha_i, \rank
  \beta_i)=(i,i-1), (i-1,i) \) or \( (i-1,i-1) \).

  If \( \beta_i \) is of maximal rank, we can put it in the standard
  form of Proposition\nobreakspace \ref {prop:betasurj}, and now \( \alpha_i \) is a \( (i+1)
  \times i \) matrix such that the bottom \( i \times i \) block is \(
  \alpha_{i-1} \beta_{i-1}- \lambda_i^\C I_{i \times i} \). The
  vanishing of \( \Phi_i \) just says that the associated minor
  determinant is zero.  The condition that \( \alpha_i \) is of rank
  \( i-1 \) means also that the other \( i \times i \) minors must
  vanish, so we get a subvariety of codimension at least two in \(
  \mu_\C^{-1}(0) \).  The case \( (i-1,i) \) follows by dualising, and
  the third case \( (i-1,i-1) \) proceeds by a similar calculation
  choosing a standard form for the rank \( i-1 \) map~\( \beta_i \).
\end{proof}

\begin{theorem}
  \label{thm:GITimp}
  The algebra of invariants \( \cO(\SL(n,\C) \times \bmf)^N \) is
  finitely generated, and the hyperk\"ahler quotient \( Q = M \hkq H
  \) of the space \( M \) of full flag quivers by \( H \) can be
  identified with the non-reductive GIT quotient
  \begin{equation*}
    (\SL(n,\C) \times \bmf) \symp N.
  \end{equation*}
\end{theorem}

\begin{proof}
  The set \( \mu_\C^{-1}(0)^\surj \) of full flag quivers satisfying
  the complex equations and with all \( \beta_i \) surjective is, by
  Lemma\nobreakspace \ref {lem:codtwo}, an open set in \( \mu_\C^{-1}(0) \) whose
  complement is of complex codimension at least two. Moreover, the
  proof of Proposition\nobreakspace \ref {prop:betasurj} shows that \( \mu_\C^{-1}(0)^\surj \)
  may be identified with \( (\SL(n,\C) \times H_\C) \times_N \bmf \).
 
  Therefore the coordinate algebra \( \cO(\mu_\C^{-1}(0))^{H_\C} \) is
  isomorphic to
  \begin{equation*}
    \cO((\SL(n,\C) \times H_\C) \times_N \bmf)^ {H_\C},
  \end{equation*}
  and hence to \( \cO(\SL(n,\C) \times \bmf)^N \).  As \(
  \mu_\C^{-1}(0) \) is an affine variety and \( H_\C \) is reductive,
  this algebra is finitely generated. Moreover
  \begin{equation*}
    Q = \mu_\C^{-1}(0) \symp H_\C 
    = \Spec\cO(\mu_\C^{-1}(0))^{H_\C} =
    \Spec\cO(\SL(n,\C) \times \bmf)^N,
  \end{equation*}
  and the last space is by definition the non-reductive quotient \(
  (\SL(n,\C) \times \bmf) \symp N \) (cf.\ the discussion in
  \S \ref {sec:towards-hyperk-impl}), so the result follows.
\end{proof}

Together with the results of \S \ref {sec:kostant-varieties}, this means
that the complex-symplectic quotients of the implosion will give the
Kostant varieties \( V_\chi \).

\begin{remark}
  \label{rem:scaling}
  The symplectic implosion \( (T^*K)_\impl \) has an action of \( \R^*
  \) induced from multiplication in the fibres of \( T^*K \).
  Similarly, we have a \( \C^* \) action on \( T^* K_\C \) given by \(
  (g, \xi) \mapsto (g, \tau \xi) \). This action commutes with the
  right \( K_\C \) action \eqref{eq:action}, and hence, for each \(
  P_S \), preserves the property of being in the zero level set for
  the complex-symplectic action of \( P_S \) and induces an action on
  the subsets \( K_\C \times_{P_S} \p_S^\circ \). Identifying these
  with the cotangent bundles of \( K_\C/P_S \), this is just scaling
  in the fibre.

  In fact, this extends to an action on the full implosion space \( Q
  \), given by just scaling the \( \beta_i \). Notice that this will
  also scale \( X = \alpha_{n-1} \beta_{n-1} \), so will induce the
  scaling in the fibre of the cotangent bundles above. We can also
  view it as the action on \( (\SL(n,\C) \times \bmf) \symp N \)
  induced by scaling the \( \bmf \) factor.
\end{remark}

\begin{theorem}
  The fixed points of the \( \C^* \) action on the universal
  hyperk\"ahler implosion \( Q \) given by scaling the \( \beta_i \)
  are represented by the quivers with \( \beta_i =0 \) for each \( i
  \).  The fixed point set may therefore be identified, by the
  discussion in \S \ref {sec:symplectic-quivers}, with the universal
  symplectic implosion for \( K = \SU(n) \).
\end{theorem}

\begin{proof}
  If \( (\alpha, \beta) \) is fixed by the action, then for all \(
  \tau \in \C^* \)  there exists \( g_\tau \in \SL \) with \(
  g_{\tau}. (\alpha, \beta) = (\alpha, \tau \beta) \). Letting \( \tau
  \rightarrow 0 \), we see that \( (\alpha, 0) \) is in the closure of
  the \( \SL \)-orbit of \( (\alpha, \beta) \), and hence in the same
  orbit by our polystability condition. It now follows that \( \beta=0
  \).
\end{proof}

\begin{proposition}
  If the universal hyperk\"ahler implosion \( Q \) for \( \SU(n) \) is
  smooth, then so is the universal symplectic implosion for \( \SU(n)
  \), and hence \( n \leqslant 2 \).
\end{proposition}

\begin{proof}
  We consider the action of the maximal compact subgroup \( S^1 \) of
  \( \C^* \).  If the hyperk\"ahler implosion were smooth, then the
  fixed point set of this circle action would also be smooth. By
  general properties of reductive group actions, this set is also the
  fixed point set of the \( \C^* \) action, which from above is just
  the universal symplectic implosion.  The result now follows from
  \cite[\S6]{Guillemin-JS:implosion} which tells us that the universal
  symplectic implosion for \( K \) is smooth if and only if the
  commutator \( [K,K] \) is a product of copies of \( \SU(2) \).
\end{proof}

\begin{remark}
  We shall see in \protect \MakeUppercase {E}xample\nobreakspace \ref {ex:SU2} that the universal hyperk\"ahler
  implosion \emph{is} smooth if \( K = \SU(2) \).
\end{remark}

\section{Geometry of the strata and torus reductions}
\label{sec:geom-strata-torus}

Let us now further investigate the geometry of the strata described in
Theorems\nobreakspace \ref {thm:strat} and\nobreakspace  \ref {thm:str-struct} and consider some examples.  We
particularly focus on the stratification of the universal
hyperk\"ahler implosion \( Q \) of \( K = \SU(n) \), where the
original quiver is a full flag, that is, \( r=n \) and \( n_j = j \)
for each \( j \). Of course, the strata will involve quivers that are
not necessarily of full flag type.

For the purposes of considering torus reductions of the implosion, we
need to focus on the \( \tT_\C \)-polystable locus; that is, the locus
where the action of \( \tH_\C = \prod_{i=1}^{r-1} \GL(n_i,\C) \) on
the quivers is polystable.  (Recall from Remark\nobreakspace \ref {rem:ident} that in the
full flag case we have a canonical identification of \( \tT_\C \) with
the maximal torus \( T_\C \) of \( K_\C \)).  Using the results of
\S \ref {sec:hyperk-quiv-diagr} and \S \ref {sec:structure-strata} we can
relate these loci to open subsets of the cotangent bundles \(
\SL(n,\C) \times_{P_S} \p_S^\circ \) of \( \SL(n,\C)/P_S \) where \(
[P,P] \leqslant P_S \leqslant P \) for suitable parabolics \( P \).

In particular, we shall first take the parabolic \( P \) to be the
Borel subgroup \( B \) with \( P_S = [B,B] = N \), and look at the
subset \( \SL(n,\C) \times_N \bmf \) in the implosion, which as we
have seen corresponds to full flag quivers with all \( \beta_i \)
surjective.

In order to analyse the \( N \) action, let us write
\begin{equation*}
  \kf_\C = \tf_\C \oplus
  \bigoplus_{\alpha \in \Delta_{+}} \kf_{\alpha} \oplus \bigoplus_{\alpha
  \in \Delta_{+}} \kf_{-\alpha},
\end{equation*}
where \( \kf_{\alpha} \) are the root spaces and \( \Delta_{+} \) the
set of positive roots.  So \( [\kf_{\alpha}, \kf_{\beta}] \subset
\kf_{\alpha + \beta} \), taking \( \kf_0 \) to be \( \tf_\C \) and \(
\kf_{\gamma} \) to be zero if \( \gamma \) is not a root. Recall also
that \( \kf_{\alpha} \) and \( \kf_{\beta} \) are Killing-orthogonal
if \( \alpha + \beta \neq 0 \).  We are using the pairing \( (A,B)
\mapsto \tr(AB) \) to identify \( \kf \) and \( \kf_\C \) with their
duals.

We can take
\begin{equation*}
  \n = \bigoplus_{\alpha \in \Delta_{+}} \kf_{\alpha}
\end{equation*}
and
\begin{equation*}
  \n^\circ = \bmf = \tf_\C \oplus \bigoplus_{\alpha \in \Delta_{+}}
  \kf_{\alpha}. 
\end{equation*}
Let \( X_{(\alpha)} \) denote the component of \( X \) in \(
\kf_{\alpha} \).  So if \( n \in \n \) and \( X \in \n^\circ = \bmf \)
then \( [n,X]_{(\alpha)} \) is \( [n_{(\alpha)}, X_{(0)}] \) plus
terms \( [n_{(\beta)}, X_{(\gamma)}] \) where \( \beta + \gamma =
\alpha \) and \( \beta, \gamma >0 \).  Moreover \(
[n_{(\alpha)},X_{(0)}] \) is just \( -\alpha(X_{(0)}) n_{(\alpha)} \).
 
The adjoint action of \( \exp(n) \in N \) on \( X \in \kf_{\alpha} \) is
\begin{equation*}
  X \mapsto X + [n,X] + \text{terms in higher  iterated brackets}.
\end{equation*}
So
\begin{equation*}
  (\exp(n) X)_{(0)} = X_{(0)}
\end{equation*}
and
\begin{equation*}
  (\exp(n) X)_{(\alpha)} = X_{(\alpha)} - \alpha(X_{(0)}) n_{(\alpha)}
  + \dotsb 
\end{equation*}
where \( \dotsb \) denotes terms in \( n_{(\beta)}, X_{(\gamma)} \)
with \( \beta \) (respectively \( \gamma \)) ranging over positive
roots less than \( \alpha \) (respectively \( 0 \) or positive roots
less than \( \alpha \)).

This means that, provided \( \alpha \) does not vanish on the Cartan
component of \( X \), we may work up inductively through the root
spaces, starting with the lowest, finding \( n_{(\alpha)} \) such that
\( (\exp(n) X)_{(\alpha)} =0 \). Moreover these \( n_{(\alpha)} \) are
uniquely determined.

So if \( X_{(0)} \) lies in the complement of the union of the zero
loci of the roots then the \( N \)-orbit through \( X \) contains a
unique element in the chosen Cartan algebra \( \tf_\C \). (See
\protect \MakeUppercase {E}xample\nobreakspace \ref {ex:SU3} for a concrete example in the \( \SU(3) \) case).

We see that if the eigenvalues of \( X \in \bmf \) are all distinct,
then the corresponding part of \( \SL(n,\C) \times_N \bmf \) may be
identified with \( \SL(n,\C) \times \tf_\C^\reg \), where \(
\tf_\C^\reg \) is the complement of the union of the zero loci for the
roots in \( \tf_\C \); that is, the set of diagonal matrices in \(
\sln(n,\C) \) with distinct entries. Note that in the setup of
Proposition\nobreakspace \ref {prop:betasurj}, this amounts to standardising the \( \alpha_i \) so
that the only non-zero entries are in position \( (j+1,j) \) for \( j=1,
\dots, i \).  The fact that the eigenvalues of \( X \) are distinct
now implies that all \( \alpha_i \) are injective.

\begin{proposition}
  \label{prop:Cartanreg}
  The implosion \( Q \) contains \( \SL(n,\C) \times \tf_\C^\reg \)
  as an open dense subset.
\end{proposition}

\begin{remark}
  Notice that (for \( K = \SU(n) \)) \( K_\C \times \tf_\C^\reg \) is
  in this sense a complex-symplectic analogue of the product of \( K
  \) with the interior of the Weyl chamber, which is the open stratum
  in the symplectic implosion of \( T^*K \). However in our case \(
  \SL(n,\C) \times \tf_\C^\reg \) is a proper subset of \( \SL(n,\C)
  \times_N \bmf_{*} \), and hence of the open hyperk\"ahler stratum,
  because configurations with equal eigenvalues may occur in \(
  Q^\hks \).
\end{remark}

Let us now look at strata corresponding to more general parabolics.
Our discussion of the diagonal entries of \( X \) in
Proposition\nobreakspace \ref {prop:betasurj} shows that the eigenvalue \( \kappa_i \) of the
trace-free part of \( X \) occurs at least \( k_{i-1} \) times (\(
i=1, \dots, r \) with the convention that \( \kappa_r =0 \)). We say
\enquote{at least} because if \( \lambda^\C_i + \lambda_{i+1}^\C +
\dots \lambda^\C_j \) is zero for \( i\leqslant j \) then \( \kappa_i
\) and \(\kappa_{j+1} \) will be equal.

Hence, given a collection of eigenvalues \( \kappa_j \) for~\( X \),
the maximum value of \( k_{i-1} \) compatible with this collection is
the multiplicity of the corresponding~\( \kappa_i \).

\begin{example}
  \label{ex:regss}
  In particular if the eigenvalues of~\( X \) are all distinct, then
  the \( k_i \) are all~\( 1 \) and we are in the full flag case where
  \( n_i = i \) for all~\( i \).  This means we are in \( \SL(n,\C)
  \times \tf_\C^\reg \), which is open and dense in \( \SL(n,\C)
  \times_N \bmf \).

  Performing hyperk\"ahler reduction by the maximal torus \( T \) is
  equivalent to fixing the \( \tf_\C \)-component of \( \bmf \), i.e.\
  fixing \( X_{(0)} \), and quotienting by \( T_\C \). If we reduce at
  a level in \( \tf_\C^\reg \) then the above discussion shows the
  resulting space is the semisimple orbit \( \SL(n,\C)/T_\C \).
\end{example}

Now let us consider reduction at a level \( X_{(0)} \) where a root
vanishes. If \( \alpha(X_{(0)}) \neq 0 \), then as above we can set \(
(\exp(n) X)_{(\alpha)} \) to zero. That is, we can use the \( N \)
action to move \( X \) into the subalgebra \( \tf_\C \oplus \n_1 \),
where
\begin{equation*}
  {\n_1} = \bigoplus_{\substack{\alpha \in \Delta_+ \\
  \alpha(X_{(0)})=0}} \kf_\alpha.
\end{equation*}
We have left a residual action of \( N_1 \), the unipotent group with Lie
algebra \( \n_1 \).

\begin{example}
  \label{ex:ss}
  In particular, suppose \( X \) has distinct eigenvalues \( \sigma_1,
  \dots, \sigma_s \) with multiplicities \( m_1, \dots, m_s \) and let
  us take the maximum possible \( k_i \) compatible with this: that
  is, we take \( r = s \) and \( k_i = m_{i+1} \), \( i = 0, \dots,
  r-1 \). Let \( P \) denote the associated parabolic, whose
  commutator contains the maximal unipotent \( N \). As we are
  concerned with the quotients of the strata by the full torus \(
  \tT_\C \) we shall assume in the following discussion that \(
  P_S = [P,P] \); this will not result in any loss of generality as
  regards the quotient, as \( P_S \) is an extension of \( [P,P] \) by
  a subtorus of \( \tT_\C \).

  We can now use part of the \( N \) action to kill the entries of \(
  X \) in positions \( (i,j) \) where \( X_{ii} \neq X_{jj} \).  But
  our choice of \( k_i \) means the other non-diagonal elements of \(
  X \) are zero, as \( X \) has to be in the annihilator \( [\p,
  \p]^\circ \). The remaining part of \( [P,P] \) that we have not
  used is just \( \prod_{i=1}^{r-1}\SL(k_i,\C) \).

  So we obtain \( \SL(n ,\C) \times \tf_\C^{(\sigma, m)} /
  \prod_{i=1}^{r-1}\SL(k_i,\C) \), where \( \tf_\C^{(\sigma, m)} \)
  denotes the subset of \( \tf_\C \) satisfying the above equalities
  of eigenvalues. Hyperk\"ahler reduction now fixes the diagonal
  entries and quotients by~\( T_\C \), hence we obtain the general
  semisimple orbit
  \begin{equation*}
    \SL(n,\C)/{\Lie S}\bigl(\prod_{i=1}^{r-1}\GL(k_i,\C)\bigr)
  \end{equation*}
  Thus we obtain the semisimple orbits (the closed stratum of the
  Kostant variety) by choosing the stratum where the \( k_i \) are the
  maximum possible given the eigenvalues. The case of distinct
  eigenvalues, where all \( k_i \) must be~\( 1 \), gives the regular
  semisimple orbit as in the preceding example.
\end{example}

In general we have that each eigenvalue multiplicity \( m_i \), \( i =
1,\dots,s \), is a sum of \( k_j \), say \( k_{i_1}+ \dots +k_{i_p}
\).  We can use part of the \( [P,P] \) action to reduce \( X \) to
block-diagonal form where we have one block for each distinct
eigenvalue \( \sigma_i \).  Moreover each block is upper triangular
with diagonal entries all equal to the eigenvalue \( \sigma_i \).

The remaining freedom in \( [P,P] \) is now also block-diagonal: each
block is the commutator of a parabolic \( P_{\sigma_i} \) in \(
\GL(m_i,\C) \).  Let us write \( P_{\sigma_i} \) as \( U_{\sigma_i}
L_{\sigma_i} \) where \( U_{\sigma_i} \) is the unipotent radical of
\( P_{\sigma_i} \) and \( L_{\sigma_i} \) is the corresponding Levi
subgroup.

The condition that \( X \) lies in \( [\p,\p]^\circ \) means that, for
each block, (transposing and dualising), the non-scalar part of that
block of \( X \) actually lies in \( \un_{\sigma_i} \). The scalar
part of the block is of course just \( \sigma_i I_{m_i \times m_i} \).

The previous example is the case when \( r=s \) and \( k_i = m_{i+1}
\), so we have one \( k_i \) for each distinct eigenvalue.  Now \(
P_{\sigma_i} = \GL(m_i,\C) \) so the unipotent \( U_{\sigma_i} \) is
scalar and hence \( X \) is scalar on each block. Moreover the
remaining freedom in \( [P,P] \) is just the product of the Levi
subgroups \( L_{\sigma_i} = \SL(k_{i-1},\C) \), in agreement with the
results of that example.

\bigbreak For some concrete low-dimensional examples consider the
following.

\begin{example}
  \label{ex:SU2}
  Let us take \( K = \SU(2) \).

  (i) The quivers we have to consider are now of the form
  \begin{equation*}
    \C \stackrel[\beta]{\alpha}{\rightleftarrows} \C^2
  \end{equation*}
  so we have \( M= \HH^2 = \Hom (\C,\C^2) \oplus \Hom(\C^2,\C)
  \). In the terminology of \S \ref {sec:hyperk-quiv-diagr} the group \(
  \tH \) is \( \Un(1) \) but its commutator \( H \) and the associated
  complex group \( \SL \) are the trivial groups \( \SU(1) \) and \(
  \SL(1,\C) \).  There is therefore no moment map equation and no
  stability condition for these groups, and the implosion \( Q \) is
  just \( \HH^2 \). We have a hyperk\"ahler action of the torus \( T =
  \Un(1) \).
  
  We can decompose the implosion into 4 subsets.  The smallest one \(
  {\mathcal S}_{\textup{bottom}} \) is when \( \alpha, \beta \) are
  both zero, so \( \C = \ker \alpha \oplus \im \beta \). At the other
  extreme, we let \( {\mathcal S}_0 \) be the set of quivers when \(
  \alpha \) is injective and \( \beta \) surjective, so again this
  direct sum condition holds.  This subset is \( (\C^2 \setminus \{ 0
  \}) \times (\C^2 \setminus \{ 0 \}) \).  We also have a subset \(
  {\mathcal S}_1 \) with \( \beta \) surjective but \( \alpha \) not
  injective (so zero), and a subset \( {\mathcal S}_2 \) with \(
  \alpha \) injective and \( \beta \) not surjective (so zero). Both
  these subsets are isomorphic to \( \C^2 \setminus \{ 0 \} \).

  We observe that \( {\mathcal S}_0 \) and \( {\mathcal
  S}_{\textup{bottom}} \) are sets of the form described in
  Proposition\nobreakspace \ref {prop:betasurj2}, corresponding to choosing the parabolic \( P
  \) to be the Borel \( B \) and the full group \( \SL(2,\C) \),
  respectively.  To see this, we can take the unipotent subgroup \( N
  \) of \( K_\C \) to be the group of upper triangular matrices
  \begin{equation*}
    N = \left\{
      \begin{pmatrix}
        1 & t \\
        0 & 1
      \end{pmatrix}
      : t \in \C \right\}.
  \end{equation*}
  We must consider \( \SL(2,\C) \times \n^\circ \), where
  \begin{equation*}
    \n^\circ = \bmf = \left\{
      \begin{pmatrix}
        a  &  b \\
        0 & -a
      \end{pmatrix}
      : a, b \in \C \right\}.
  \end{equation*}
  Writing an element of \( \SL(2,\C) \) as \(
  \begin{psmallmatrix}
    q_{11}  &  q_{12} \\
    q_{21} & q_{22}
  \end{psmallmatrix} \), the \( N \) action is
  \begin{align*}
    q_{12} & \mapsto  q_{12} - t q_{11} \\
    q_{22} & \mapsto  q_{22} -t q_{21}  \\
    b & \mapsto b - 2at
  \end{align*}
  while \( q_{11}, q_{21}, a \) are invariant. Also \( q_{11}q_{22} -
  q_{12} q_{21} =1 \).

  According to the discussion in \S \ref {sec:structure-strata}, \(
  {\mathcal S}_0 \) will be just the open dense set \( \SL(2,\C)
  \times_N \bmf_* \) in \( \SL(2,\C) \times_N \bmf \), where \( X \in
  \bmf \) is non-zero (that is, \( a,b \) are not both zero).

  If \( a \) is non-zero then taking the \( N \) quotient is
  equivalent to setting \( b =0 \). We obtain a set \( \SL(2,\C)
  \times (\C \setminus \{0\}) \), which may be identified with \(
  (\C^2 \setminus \{0 \}) \times \C \times (\C \setminus \{0\}) \)
  (viewing the \( \C^2 \setminus \{0 \} \) factor as the first column
  of the matrix in \( \SL(2,\C)) \).

  If \( a \) is zero then the \( N \) action on the \( \bmf \) factor
  is trivial, and we obtain \( (\SL(2,\C) / N) \times \{ X \in \bmf :
  a=0, \ b \neq 0 \} \), which is just \( (\C^2 \setminus \{0 \})
  \times (\C \setminus \{0 \}) \).  Identification of \( (\SL(2,\C) /
  N \) with \( \C^2 \setminus \{0 \} \) follows, for example, from the
  Iwasawa decomposition---this space is of course the open stratum of
  the symplectic implosion for \( \SU(2) \)).

  The whole of \( {\mathcal S}_0 \), then, may be viewed as \( (\C^2
  \setminus \{0 \}) \times (\C^2 \setminus \{0 \}) \) in accordance
  with the quiver picture. On the other hand \( {\mathcal
  S}_{\textup{bottom}} \) has \( X=0 \) so is just \(
  \SL(2,\C)/{\SL(2,\C)} \); that is, a point, again agreeing with the
  quiver description above.

  Note that if instead we took the whole of \( \SL(2,\C) \times_N
  \bmf \), rather than requiring \( X \neq 0 \), then we would obtain
  \( (\C^2 \setminus \{0\}) \times \C^2 \). This of course corresponds
  to the union of the two sets \( {\mathcal S}_0 \) and \( {\mathcal
  S}_1 \), that is, quivers with \( \beta \) surjective, in accordance
  with Proposition\nobreakspace \ref {prop:betasurj}.

  In terms of the hyperk\"ahler stratification in
  \S \ref {sec:strat-quiv-diagr}, the open stratum is the union \( \HH^2
  \setminus \{ 0 \} \) of \( {\mathcal S}_0,{\mathcal S}_1 \) and \(
  {\mathcal S}_2 \), which can also be viewed as the \( \SUr \) sweep
  of \( {\mathcal S}_0 \). The closed stratum is just \( {\mathcal
  S}_{\textup{bottom}} \), which is the origin. In the language of
  Proposition\nobreakspace \ref {prop:mod-embed}, the open stratum corresponds to taking \( S
  \) empty, so there is no \( \delta \), while the closed stratum
  corresponds to taking \( S = \{(1,1)\} \) with \( p = h = k = d_1 =
  1 \) and \( 0 = m_0 = m_1 < m_2 =2 \).

  From the point of view of reductions by the torus \( \Un(1) \) the
  relevant sets are \( {\mathcal S}_{\textup{bottom}} \) and \(
  {\mathcal S}_0 \) because these give the closed orbits for the
  complexified action on \( \HH^2 \), in agreement with
  Proposition\nobreakspace \ref {prop:KobS}. The union of these strata can be viewed as the set
  of pairs \( (z,w) \) in \( \C^2 \times \C^2 \) such that \( z,w \)
  are both zero or both non-zero.

  The hyperk\"ahler quotient by \( \Un(1) \) gives us the
  hyperk\"ahler structure on Kostant varieties of \( \SL(2,\C) \)
  (i.e.\ Eguchi-Hanson or the nilpotent variety) as explained below.

  \bigbreak (ii) It is also instructive, in the light of
  Theorem\nobreakspace \ref {thm:GITimp}, to consider the GIT quotient \( (\SL(2,\C)
  \times \bmf) \symp N \). Using the variables above, the invariant
  polynomials are generated by \( q_{11}, q_{21}, a \) and \( y_1 =
  2q_{22}a - q_{21}b, y_2= 2q_{12}a - q_{11}b \).  We have the
  relation
  \begin{equation}
    \label{eq:su2}
    q_{11} y_1 - q_{21} y_2 =2a,
  \end{equation}
  so \( (\SL(2,\C) \times \bmf ) \symp N \) is the affine
  hypersurface in \( \C^5 \) with equation \eqref{eq:su2}.  Now
  projection onto \( (q_{11}, q_{21}, y_1, y_2) \) gives an
  isomorphism with \( \C^4 = \HH^2 \).

  The \( \C^* \)-action of Remark\nobreakspace \ref {rem:scaling} is just scaling of \(
  a,b \), hence of \( a,y_1, y_2 \). Its fixed-point set on \( \C^4 =
  \HH^2 \) is just given by \( y=0 \) and hence is a copy of \( \C^2
  \), the symplectic implosion of \( T^*\SU(2) \).

  \medbreak The \( T_\C \) action is \( q_{i1} \mapsto s^{-1} q_{i1}
  \), \( q_{i2} \mapsto s q_{i2} \), \( b \mapsto s^2 b \) while \( a
  \) is invariant. So \( y_i \mapsto s y_i \), and equation
  \eqref{eq:su2} is preserved.

  Under the above identification with \( \C^4 \), the coordinates \(
  q_{11}, q_{21} \) scale by \( s^{-1} \) and \( y_1, y_2 \) by \( s
  \).  Note that the non-closed orbits for the \( T_\C \) action are
  therefore those lying in \( q_{11} = q_{12}=0 \) and \( y_1 = y_2=0
  \), apart from the origin which is a closed point orbit.

  Now, it is well-known that the hyperk\"ahler reduction of flat \(
  \HH^2 \) by \( \Un(1) \) at the generic level gives the
  Eguchi-Hanson metric on the semisimple orbit of \( \SL(2,\C) \). In
  terms of our variables, making the reduction is equivalent to fixing
  the value of \( a \) (i.e.\ the value of the complex-symplectic
  moment map for \( T_\C \)), and then quotienting by~\( T_\C
  \). Taking as invariant polynomials on \( \C^4 \) the expressions \(
  W = q_{11} y_1, Y = q_{11}y_2, Z= q_{21}y_1 \) and \( q_{21}y_2 \)
  (which equals \( W-2a \)), we obtain the affine surface \(
  W(W-2a)=YZ \), which is one of the complex structures for
  Eguchi-Hanson.  If we reduce \( \HH^2 \) by \( \Un(1) \) at level \(
  0 \), of course, we get the nilpotent variety \( W^2 = YZ \) of \(
  \SL(2,\C) \) with singularity at the origin.
\end{example}

\begin{example}
  \label{ex:SU3}
  Let us take \( K= \SU(3) \). From the quiver picture, we see the
  implosion is a hyperk\"ahler quotient of \( \HH^8 \) by \( \SU(2)
  \); in fact it may be realised as the Swann bundle (with origin
  adjoined) of the quaternionic K\"ahler manifold \(
  \widetilde{\operatorname{Gr}}_4 (\R^8) \) of oriented 4-planes in
  \( \R^8 \).  Note that this space is not smooth but has a conical
  singularity at the origin.

  The \( N \) action on \( \bmf = \n^\circ \) is:
  \begin{equation*}
    \begin{pmatrix}
      1 & r & s \\
      0 & 1 & t \\
      0 & 0 & 1
    \end{pmatrix}
    \colon
    \begin{pmatrix}
      a & b & c \\
      0 & d & e \\
      0 & 0 & f
    \end{pmatrix}\\
    \longmapsto
    \begin{pmatrix}
      a & b + r(d-a) & \left(\substack{c + rt(a-d) -tb \\
        + re + s(f-a)}\right) \\[1ex]
      0 & d          & e + t (f-d)                   \\
      0 & 0 & f
    \end{pmatrix}
  \end{equation*}
  where \( a+d+f=0 \), of course.

  Let us first take \( a,d,f \) distinct, so we must be in the open
  stratum.  As we are on the complement of the union of zero loci of
  the roots \( a-d, d-f, f-a \), taking the quotient by \( N \) is
  equivalent to setting \( b=c=e=0 \). So we obtain, as in
  Proposition\nobreakspace \ref {prop:Cartanreg}, a set which can be identified with \(
  \SL(3,\C) \times \tf_\C^\reg \).

  The hyperk\"ahler reduction at such a level by the action of the
  maximal torus \( T \) will be just the complex-symplectic quotient
  by the complex torus \( T_\C \).  This will be obtained by fixing
  the value of \( (a,d,f) \) and then factoring out by \( T_\C \), so
  we get \( \SL(3,\C)/T_\C \) which is a semisimple orbit for \(
  \SL(3,\C) \).

  \medbreak Now consider the case of eigenvalues \( (a,a,-2a) \) with
  \( a \neq 0 \), i.e.\ when \( a=d \neq f \).  We can obtain this
  configuration by reducing at level \( (\lambda_1^\C, \lambda_2^\C) =
  (3a,0) \) in the stratum \( Q_{(\{2,2 \},1)} \), which has \( k_0 =
  1, k_1 = 0, k_2 = 2 \), i.e.\ \( n_1 = n_2 = 1, n_3 = 3 \). So the
  parabolic \( P \) consists of matrices of the form
  \begin{equation*}
    \begin{psmallmatrix}
      * & *  & * \\
      * & *  & * \\
      0 & 0 & *
    \end{psmallmatrix}.
  \end{equation*}

  As above, we see that each \( N \)-orbit contains an element \( X
  \in \sln(3,\C) \) with \( c=e=0 \). Moreover \( b=0 \) because \( X
  \) must lie in \( [\p, \p]^\circ \) (after suitable dual
  identifications), so again we get diagonal \( X \). The remaining \(
  [P,P] \) action is that of \( \SL(2,\C) \times \{1\} \), and the
  hyperk\"ahler reduction is, as in \protect \MakeUppercase {E}xample\nobreakspace \ref {ex:ss}, the non-regular
  semisimple orbit \( \SL(3,\C)/{\Lie S(\GL(2,\C)} \times \GL(1,\C))
  \)---the closed stratum of the Kostant variety for \( (a,a-2a) \).

  The level \( (a,a,-2a) \) is also compatible with the stratum with
  \( k_0 = k_1 = k_2=1 \), and hence \( n_1=1, n_2=2, n_3=3 \): this
  is the open stratum.  Now we can make \( c=e=0 \) but \( b \) may be
  non-zero, and the residual freedom in \( N \) consists of
  block-diagonal matrices in \( N \) of block size 2 and 1.  This will
  be the open stratum of the Kostant variety for \( (a,a,-2a) \), and
  of course is not semi-simple.

  \medbreak Finally, let us consider the level \( (0,0,0) \). This is
  compatible (using Remark\nobreakspace \ref {rem:lambdaconstr}) with three strata:
  \begin{inparaenum}
  \item the open stratum with \( k_0 = k_1 = k_2 =1 \);
  \item the stratum with \( k_0 =0,k_1=1, k_2=2 \);
  \item the closed stratum with \( k_0 =k_1=0, k_2=3 \) corresponding
    to the point quiver.
  \end{inparaenum}

  On torus reduction, we will obtain the corresponding strata of the
  nilpotent variety: respectively, these are the regular stratum
  (minimum polynomial \( x^3 \)), the subregular stratum (minimum
  polynomial \( x^2 \)), and the zero element, i.e.\ the semisimple
  stratum.
\end{example}

\end{document}